\documentclass[10pt]{elsarticle}
\usepackage{amssymb,amsfonts,amsmath,mathrsfs,mathtools}
\usepackage{graphicx,epsfig,subfigure}
\usepackage{bm}
\usepackage[margin=1in,papersize={8.5in,11in}]{geometry}
\usepackage{times}
\usepackage{multirow}
\usepackage{color}
\usepackage[ruled,vlined]{algorithm2e}

\def\vs{\vspace{0.2cm}}

\DeclareMathOperator*{\argmin}{\arg\!\min}

\newtheorem{lemma}{\bf Lemma}[section]

\newtheorem{proposition}{\bf Proposition}[section]

\newenvironment{proof}{{\noindent \bf \em Proof:}}{\hfill$\square$}

\renewcommand{\top}{\text{T}}

\title{Rank-adaptive tensor methods for high-dimensional nonlinear PDEs}

\begin{document}
\begin{frontmatter}

\author[ucsc]{Alec Dektor}
\author[ucsc]{Abram Rodgers}
\author[ucsc]{Daniele Venturi\corref{correspondingAuthor}}
\ead{venturi@ucsc.edu}

\address[ucsc]{Department of Applied Mathematics, University of California Santa Cruz\\ Santa Cruz (CA) 95064}

\cortext[correspondingAuthor]{Corresponding author}

\journal{ArXiv}

\begin{abstract}
We present a new rank-adaptive tensor method to compute the numerical solution of high-dimensional nonlinear PDEs. The method combines functional tensor train (FTT) series expansions, operator splitting time integration, and a new rank-adaptive algorithm based on a thresholding criterion that limits the component of the PDE velocity vector normal to the FTT tensor manifold. This yields a scheme that can add or remove tensor modes adaptively from the PDE solution as time integration proceeds.  The new method is designed to improve computational efficiency, accuracy and robustness in numerical integration of high-dimensional problems. In particular, it overcomes well-known computational challenges associated with dynamic tensor integration, including low-rank modeling errors and the need to invert covariance matrices of tensor cores at each time step. Numerical applications are presented and discussed for linear and nonlinear advection problems in two dimensions, and for a four-dimensional Fokker-Planck equation.
\end{abstract}
\end{frontmatter}

\section{Introduction}
\label{sec:intro}
High-dimensional partial differential equations (PDEs) arise 
in many areas of engineering, physical sciences and 
mathematics. Classical examples are equations involving 
probability density functions (PDFs) such as 
the Fokker--Planck equation \cite{Risken}, the 
Liouville equation \cite{Venturi_PRS,HeyrimJCP_2014,HeyrimPRS_2014}, and the Boltzmann equation \cite{cercignani1988,dimarco2014,TensorBGK}.  
More recently, high-dimensional PDEs 
have also become central to many new areas of 
application such as optimal mass transport \cite{Osher2019,Villani}, 
random dynamical systems \cite{Venturi_MZ,Venturi_PRS}, 
mean field games \cite{Weinan2019,Ruthotto2020}, and 
functional-differential equations 
\cite{VenturiSpectral,venturi2018numerical}.
Computing the numerical solution to high-dimensional PDEs 
is an extremely challenging problem which has attracted 
substantial research efforts in recent years.
Techniques such as sparse collocation 
methods \cite{Bungartz,Chkifa,Barthelmann,Foo1,Akil}, 
high-dimensional model representations 
\cite{Li1,CaoCG09,Baldeaux},  
deep neural networks \cite{Raissi,Raissi1,Zhu2019}, 
and numerical tensor methods 
\cite{khoromskij,Bachmayr,rodgers2020stability,parr_tensor,
Hackbusch_book,Vandereycken_2019} 
were proposed to mitigate the exponential growth 
of the degrees of freedom, the computational cost and 
the memory requirements.

In this paper, we build upon our recent work on dynamical 
tensor approximation \cite{Dektor_2020,Dektor_dyn_approx},
and develop new rank-adaptive temporal integrators 
to compute the numerical solution of high-dimensional  
initial/boundary value problems of the form 
\begin{align}
\begin{cases}
\displaystyle\frac{\partial u(\bm x, t) }{\partial t} = G(u(\bm x, t)), 
\qquad  \vspace{0.1cm} \\
u(\bm x,0) = u_0(\bm x),
\end{cases}
\label{nonlinear-ibvp} 
\end{align}
where $\bm x\in \Omega \subseteq \mathbb{R}^d$  
($\Omega$ compact, $d\geq 1$), and $G$ is a nonlinear 
operator which may take into account boundary conditions. 
A well-known challenge of dynamic tensor 
approximations to \eqref{nonlinear-ibvp} is that 
the curvature of the tensor manifold in which we 
compute the PDE solution is inversely proportional 
to the energy of the tensor modes. This means that 
the smaller the energy of the tensor modes the higher 
the curvature. Hence, to integrate a solution 
characterized by tensor modes with 
a wide range of energies one has to consider
time stepping schemes that can effectively 
handle geometric features associated 
with the curvature of the manifold. 
In projection-based approaches \cite{Lubich_2007,Lubich_2008,Dektor_2020,Dektor_dyn_approx}
the computational challenge posed by the 
curvature of the tensor manifold translates into 
the need to invert the positive semi-definite 
covariance matrix of the tensor cores at each 
time step. 
A time-integration scheme constructed in this way 
may become numerically unstable in the presence of 
tensor modes with small energy, or even singular 
when modes with zero energy are 
present (e.g., at a time instant in which we increase
the tensor rank by adding a mode with zero energy).
To mitigate this problem, Babaee {\em et al.} \cite{robust_do/bo}
introduced a matrix pseudo-inverse approximation method 
that can handle potential singularities in the covariance 
matrices of the tensor cores, in particular when 
adding modes with zero energy to the tensor series 
expansion of the PDE solution.

A mathematically rigorous framework to integrate 
dynamical tensors over manifolds with arbitrary 
curvature was developed by Lubich {\em et. al} in 
\cite{Lubich_2015,Lubich_2016,Lubich_2018}. 
The key idea is to integrate the evolution 
equation generating the tensor dynamics 
using operator splitting schemes, e.g., the 
Lie-Trotter or the Strang time integrators (see 
\cite{Lubich_2016,Lubich_2018} 
for details). 
This results in a scheme that does not suffer from the 
curvature of the tensor manifold, and even provides an 
exact representation in the presence of tensor 
modes with zero energy. 
The numerical method presented in this work 
combines all these features, i.e., functional tensor 
train (FTT) series expansions, operator splitting time 
integration, and a new rank-adaptive algorithm to add 
and remove tensor modes from the PDE solution based 
on a thresholding criterion that limits the component 
of the  velocity vector normal to the FTT tensor manifold.

This paper is organized as follows. 
In section \ref{sec:tensor_manifolds} we briefly review 
finite-rank functional tensor train (FTT) expansions of 
high-dimensional functions. 
In section \ref{sec:approx_on_manifolds} we discuss 
dynamic tensor approximation of nonlinear PDEs of the form 
\eqref{nonlinear-ibvp} and develop robust temporal 
integration schemes based on operator splitting methods. 
We also discuss step-truncation algorithms 
\cite{rodgers2020step-truncation,rodgers2020stability} 
and prove that dynamic tensor approximation and 
step-truncation are at least order one consistent 
to one another. 
In section \ref{sec:adaptive_rank} we develop new 
rank-adaptive time integrators on rank-structured 
FTT tensor manifolds and prove that the resulting scheme 
is consistent.
In section \ref{sec:numerics} we present and discuss 
various numerical applications of the proposed 
rank-adaptive tensor method, and demonstrate its 
accuracy and computational efficiency. 
The main findings are summarized in section \ref{sec:summary}.

\section{The manifold of fixed-rank FTT tensors}
\label{sec:tensor_manifolds}
Let us consider the weighted Hilbert 
space\footnote{The approximation theory presented in 
section \ref{sec:tensor_manifolds} may be applied 
to more general Hilbert spaces, e.g., to the 
Sobolev space $W^{2,p}$ \cite{Dektor_2020}
and to tree-based tensor manifolds in Banach 
spaces \cite{Falco_2016}.} 
\begin{equation} 
\label{hilbert_space}
H = L^2_{\mu}(\Omega),
\end{equation} 
where 
$\Omega \subseteq \mathbb{R}^d$ is a separable domain 
such as a $d$-dimensional flat torus $\mathbb{T}^d$ or a 
Cartesian product of $d$ real intervals $\Omega_i=[a_i,b_i]$
\begin{equation}
\begin{aligned}
\Omega &= \bigtimes_{i=1}^d \Omega_i, 
\end{aligned}
\label{Omega}
\end{equation}
and $\mu$ is a finite product measure on 
$\Omega$
\begin{equation}
\mu(\bm x) = \prod_{i=1}^d \mu_i(x_i).
\end{equation}
Let $\tau$ be the counting measure 
on $\mathbb{N}$. 
Each element $u \in L^2_{\mu}(\Omega)$ admits a 
functional tensor train (FTT) expansion 
of the form 
\begin{equation}
\label{FTT}
u(\bm x) = {\sum_{\alpha_0 = 1}^{\infty} 
\sum_{\alpha_1 = 1}^{\infty} \cdots  \sum_{\alpha_{d}=1}^{\infty}}  \sqrt{\lambda(\alpha_{d-1})}
\psi_1(1;x_1;\alpha_1) \psi_2(\alpha_1;x_2;\alpha_2) \cdots 
\psi_d(\alpha_{d-1};x_d;1),
\end{equation}
where {$\{\psi_i(\alpha_{i-1};x_i;\alpha_i)\}_{\alpha_{i}}$ 
is an orthonormal basis 
for the space $L^2_{\tau \times \mu_i}(\mathbb{N} \times \Omega_i)$. 
It can be shown that $\psi_i(\alpha_{i-1};x_i;\alpha_i)$ 
are eigenfunctions for a self-adjoint compact operator and 
%
%
$ \lambda(1)\geq \lambda(2) \geq \ldots \geq 0$ 
is a sequence of real numbers converging to zero (see \cite{Bigoni_2016,Dektor_dyn_approx} for more details).}
By truncating \eqref{FTT} so that only the largest 
singular values are 
retained, we obtain the approximation of $u(\bm x)$
\begin{equation}
\label{FTT_finite}
{u_{\bm r}}(\bm x) = \sum_{\alpha_0 = 1}^{r_0} 
\sum_{\alpha_1 = 1}^{r_1} \cdots  \sum_{\alpha_{d}=1}^{r_{d}} 
\sqrt{\lambda(\alpha_{d-1})}
\psi_1(\alpha_0;x_1;\alpha_1) \psi_2(\alpha_1;x_2;\alpha_2) \cdots 
\psi_d(\alpha_{d-1};x_d;\alpha_d),
\end{equation}
where $\bm r = (r_0, r_1, \ldots, r_{d-1}, r_d)$ is the {FTT} rank. 
It is convenient to write \eqref{FTT_finite} in a more 
compact form as 
\begin{equation}
\label{FTT_core_rep}
    {u_{\bm r}}(\bm x) = \bm \Psi_1(x_1) \bm \Psi_2(x_2) \cdots \sqrt{\bm \Lambda} \bm \Psi_d(x_d),
\end{equation}
{
where $\bm\Psi_i(x_i)$ is a $r_{i-1} \times r_i$ matrix with entries 
$\left[\bm\Psi_i(x_i)\right]_{j k} = \psi_i(j;x_i;k)$} and $\bm \Lambda$ is a diagonal matrix with entries $\lambda(\alpha_{d-1})$ ($\alpha_{d-1}=1,\ldots,r_{d-1}$). 
The matrix-valued functions $\bm\Psi_i(x_i)$ will be referred 
to as FTT {\em cores}, and we denote by $M_{r_{i-1} \times r_i}(L_{\mu_i}^2(\Omega_i))$ the set of all $r_{i-1} \times r_i$ matrices with entries in 
$L_{\mu_{i-1}}^2(\Omega_i)$. 
To simplify notation even more, we will often 
suppress explicit tensor core dependence on the spatial variable $x_i$, allowing us 
to simply write $\bm\Psi_i=\bm\Psi_i(x_i)$ and $\psi_i(\alpha_{i-1},\alpha_i)=\psi_i(\alpha_{i-1};x_i;\alpha_i)$ as the spatial dependence is indicated 
by the tensor core subscript.

\subsection{{Orthogonalization and truncation of FTT tensors}}
\label{sec:orthogonalization}
For any tensor core $\bm \Psi_i \in M_{r_{i-1} \times r_i}
(L_{\mu_i}^2(\Omega_i))$ we define the matrix 
\begin{equation}
\langle \bm \Psi^{{\top}}_ i\bm \Psi_i \rangle_{i} \in M_{r_i \times r_i}(\mathbb{R}) 
\label{Cdef}
\end{equation}
with entries\footnote{The averaging operation in \eqref{Cdef1} can be viewed as a 
an inner product on the space $\displaystyle\bigtimes_{n=1}^{r_{i-1}} L_{\mu_i}^2(\Omega_i)$.}
{
\begin{align}
\left\langle \bm \Psi_i^{\top} \bm \Psi_i \right\rangle_i(j,k) = 
\sum_{p=1}^{r_{i-1}} \int_{\Omega_i} \psi_i(p;x_i;j) \psi_i(p;x_i;k) d \mu_i(x_i).
\label{Cdef1}
\end{align}}
The FTT representation \eqref{FTT_core_rep} is given in terms 
of FTT cores $\bm \Psi_i$ satisfying\footnote{{
Equation \eqref{ortho_cores} follows immediately from the orthonormality of $\{\psi_i(\alpha_{i-1};x_i;\alpha_i)\}_{\alpha_{i}}$ 
relative to the inner product in $L^2_{\tau \times \mu_i}(\mathbb{N} \times \Omega_i)$.} }
\begin{equation}
\label{ortho_cores}
\begin{aligned}
\left\langle \bm \Psi_i^{\top}\bm \Psi_i\right\rangle_i &= 
\bm I_{r_i \times r_i}, \qquad i = 1,\ldots, d-1, \\ 
\quad \left\langle \bm \Psi_d\bm \Psi_d^{\top} \right\rangle_d 
&= \bm I_{r_{d-1} \times r_{d-1}}.
\end{aligned}
\end{equation}
Other orthogonal representations can be computed, e.g., 
based on recursive QR decompositions.
To describe different orthogonalizations of FTT tensors, 
let $\bm \Psi_i \in M_{r_{i-1} \times r_i}
(L^2_{\mu_i}(\Omega_i))$ and consider each column 
of $\bm \Psi_i$ as a vector in 
$\displaystyle\bigtimes_{n=1}^{r_{i-1}} L_{\mu_i}^2(\Omega_i)$.
Performing an orthogonalization process (e.g. Gram-Schmidt) 
on the columns of the FTT core $\bm \Psi_i$ relative to the 
inner product \eqref{Cdef} yields a QR-type decomposition 
of the form
\begin{equation}
\label{functional_QR}
\bm \Psi_i = \bm Q_i \bm R_i,
\end{equation}
where $\bm Q_i$ is an $r_{i-1} \times r_i$ matrix with elements 
in $L^2_{\mu_i}(\Omega_i)$ satisfying 
$\left\langle \bm Q_i^{\top}\bm Q_i \right\rangle_i = 
\bm I_{r_{i} \times r_i}$, and $\bm R_i$ 
is an upper triangular $r_i \times r_i$ matrix with real entries.
Next consider an arbitrary FTT tensor 
${u_{\bm r}} = \bm \Psi_1 \bm \Psi_2 \cdots \bm \Psi_d$, 
{
where the matrix $\langle \bm \Psi_i^{\top} \bm \Psi_i \rangle_i$ may be singular}.
For notational convenience, we define the partial products 
\begin{equation}
    \bm \Psi_{\leq i} = \bm \Psi_1 \cdots \bm \Psi_i\qquad 
    \text{and}
     \qquad \bm \Psi_{> i} = 
    \bm \Psi_{i+1} \cdots \bm \Psi_d.
\end{equation}
One way to orthogonalize ${u_{\bm r}}$ is by performing QR 
decompositions recursively from left to right as we will 
now describe. Begin by decomposing $\bm \Psi_1$ as 
\begin{equation} 
\label{core1_left_orth}
\begin{aligned}
\bm \Psi_1 = \bm Q_1 \bm R_1, \qquad &
\bm Q_1 \in M_{r_0 \times r_1}(L^2_{\mu_1}(\Omega_1)), 
\quad \left\langle \bm Q_1^{\top}
\bm Q_1 \right\rangle_1 = 
\bm I_{r_1 \times r_1}, \\ 
&\bm R_1 \in M_{r_1 \times r_1}(\mathbb{R}) 
\text{ is upper triangular}.
\end{aligned}
\end{equation}
Now we may write ${u_{\bm r}} = \bm Q_1 \bm R_1 \bm \Psi_2 
\cdots \bm \Psi_d$.
Next, perform another QR decomposition 
\begin{equation} 
\label{core2_left_oth}
\begin{aligned}
\bm R_1 \bm \Psi_2 = \bm Q_2 \bm R_2, \qquad &
\bm Q_2 \in M_{r_1 \times r_2}(L^2_{\mu_2}(\Omega_2)), \quad 
\left\langle \bm Q_2^{\top} \bm Q_2 \right\rangle_{2} = 
\bm I_{r_2 \times r_2}, \\
&\bm R_2 \in M_{r_2 \times r_2}(\mathbb{R}) \text{ is upper triangular}.
\end{aligned} 
\end{equation}
Proceeding recursively in this way we obtain a representation 
for ${u_{\bm r}}$ of the 
form 
\begin{equation}
\label{left_orth}
{u_{\bm r}} = \bm Q_1 \cdots \bm Q_{d-1} \bm Q_d \bm R_d,
\end{equation} 
where each 
$\bm Q_i \in M_{r_{i-1} \times r_i}(L^2_{\mu_i}(\Omega_i))$ 
satisfies $\left\langle \bm Q_i^{\text{T}}
\bm Q_i\right\rangle_{i} = 
\bm I_{r_i \times r_i}$. We refer to such a representation as a 
left orthogonalization of ${u_{\bm r}}$. We may stop 
orthogonolizing at any step in the recursive process 
to obtain the partial left orthogonalization 
\begin{equation} 
\label{partial_left_orth}
{u_{\bm r}} = \bm Q_{\leq i} \bm R_i \bm \Psi_{>i}.
\end{equation}
Similar to orthogonalizing from the left, we may also orthogonalize ${u_{\bm r}}$ from 
the right. To do so, begin by performing a QR decomposition 
{
\begin{equation}
\label{cored_right_orth}
\begin{aligned}
\bm \Psi_d^{\top} = \bm K_d \bm W_d, \qquad &\bm K_d \in M_{r_d \times r_{d-1}}(L^2_{\mu_d}(\Omega_d)), \quad
\left\langle \bm K_d^{\top} \bm K_d \right\rangle_{d} = 
\bm I_{r_{d-1} \times r_{d-1}}, \\ 
&\bm W_d \in M_{r_{d-1} \times r_{d-1}}(\mathbb{R}) \text{ is upper triangular}. 
\end{aligned}
\end{equation}
A substitution of \eqref{cored_right_orth} into 
\eqref{FTT_core_rep} yields the expansion $u_{\bm r} = \bm \Psi_1 \cdots \bm \Psi_{d-1} \bm W_d^{\top} \bm K_d^{\top}$. 
Next perform a QR decomposition 
\begin{equation} 
\label{core_d-1_right_orth}
\begin{aligned} 
\bm W_d \bm \Psi_{d-1}^{\top} = \bm K_{d-1} \bm W_{d-1}, \qquad &\bm K_{d-1} \in M_{r_{d-1} \times r_{d-2}}(L^2_{\mu_{d-1}}(\Omega_{d-1})), \quad \left\langle \bm K_{d-1}^{\top}\bm K_{d-1} \right\rangle_{d-1} = \bm I_{r_{d-2} \times r_{d-2}}, \\
&\bm W_{d-1} \in M_{r_{d-2} \times r_{d-2}}(\mathbb{R}) \text{ is upper triangular}.
\end{aligned}
\end{equation}
Proceeding recusively in this way we obtain the right orthogonalization 
\begin{equation} 
\label{right_orth}
u_{\bm r} = \bm W_1^{\top} \bm K_1^{\top} \cdots \bm K_d^{\top}. 
\end{equation}
We may have stopped the orthogonalization process 
at any point to obtain the partial right orthogonalization 
\begin{equation}
\label{partial_right_orth}
u_{\bm r} = \bm \Psi_{\leq i} \bm W_{i+1}^{\top} \bm K_{> i}^{\top}.
\end{equation}
It is also useful to orthogonalize from the left and 
right to obtain expansions of the form 
\begin{equation} 
\label{left_right_orth}
u_{\bm r} = \bm Q_{\leq i} \bm R_i \bm W_{i+1}^{\top} 
\bm K_{> i}^{\top},
\end{equation}
where the rank of the matrix $\bm R_i \bm W_{i+1}^{\top}$ 
is the $i$-th component of the true FTT rank of the tensor 
${u_{\bm r}}$.}

Another important operation is truncation of FTT tensors 
to smaller rank. Efficient algorithms to perform this operation 
for TT tensors can be found in \cite[section 3]{OseledetsTT}
and in \cite{AlDaas2020}. Such algorithms 
are easily adapted to FTT tensors by replacing QR 
decompositions of matrices with the QR of FTT cores 
given in \eqref{functional_QR} and SVD decomposition of 
matrices with Schmidt decompositions. In numerical 
implementations, this adaptation amounts to introducing appropriate 
quadrature weight matrices into the algorithms.

\subsection{Tangent and normal spaces of fixed-rank FTT manifolds}
Let us denote by $V_{r_{i-1} \times r_i}^{(i)}$ the set of all 
tensor cores $\bm \Psi_i \in M_{r_{i-1} \times r_{i}}(L^2_{\mu_i}(\Omega_i))$ 
with the property that the autocovariance matrices 
$\left\langle \bm \Psi_i^{\top}\bm \Psi_i \right\rangle_{i} 
\in M_{r_i \times r_i}(\mathbb{R})$ and 
$\left\langle \bm \Psi_i\bm \Psi_i^{\top} \right\rangle_{i}
\in M_{r_{i-1} \times r_{i-1}}(\mathbb{R})$ are invertible for 
$i =1,\ldots,d$.
The set 
\begin{equation}
\label{FTT_manifold}
    \mathcal{M}_{\bm r} = \{ {u_{\bm r}} \in L^2_{\mu}(\Omega):\quad {u_{\bm r}} 
= \bm \Psi_1 \bm \Psi_2 \cdots \bm \Psi_d , \quad 
\bm \Psi_i \in V_{r_{i-1} \times r_i}^{(i)}, \quad \forall i = 1,2,\ldots,d
\},
\end{equation}
consisting of fixed-rank FTT tensors, 
is a smooth Hilbert submanifold of $L^2_{\mu}(\Omega)$ 
(see \cite{Dektor_dyn_approx}).
We represent elements in the tangent space, ${T_{u_{\bm r}} \mathcal{M}_{\bm r}}$, 
of $\mathcal{M}_{\bm r}$ at 
the point ${u_{\bm r}} \in \mathcal{M}_{\bm r}$ as equivalence classes 
of velocities of continuously differentiable curves on $\mathcal{M}_{\bm r}$ 
passing through ${u_{\bm r}}$
\begin{figure}
\centerline{\includegraphics[height=6cm]{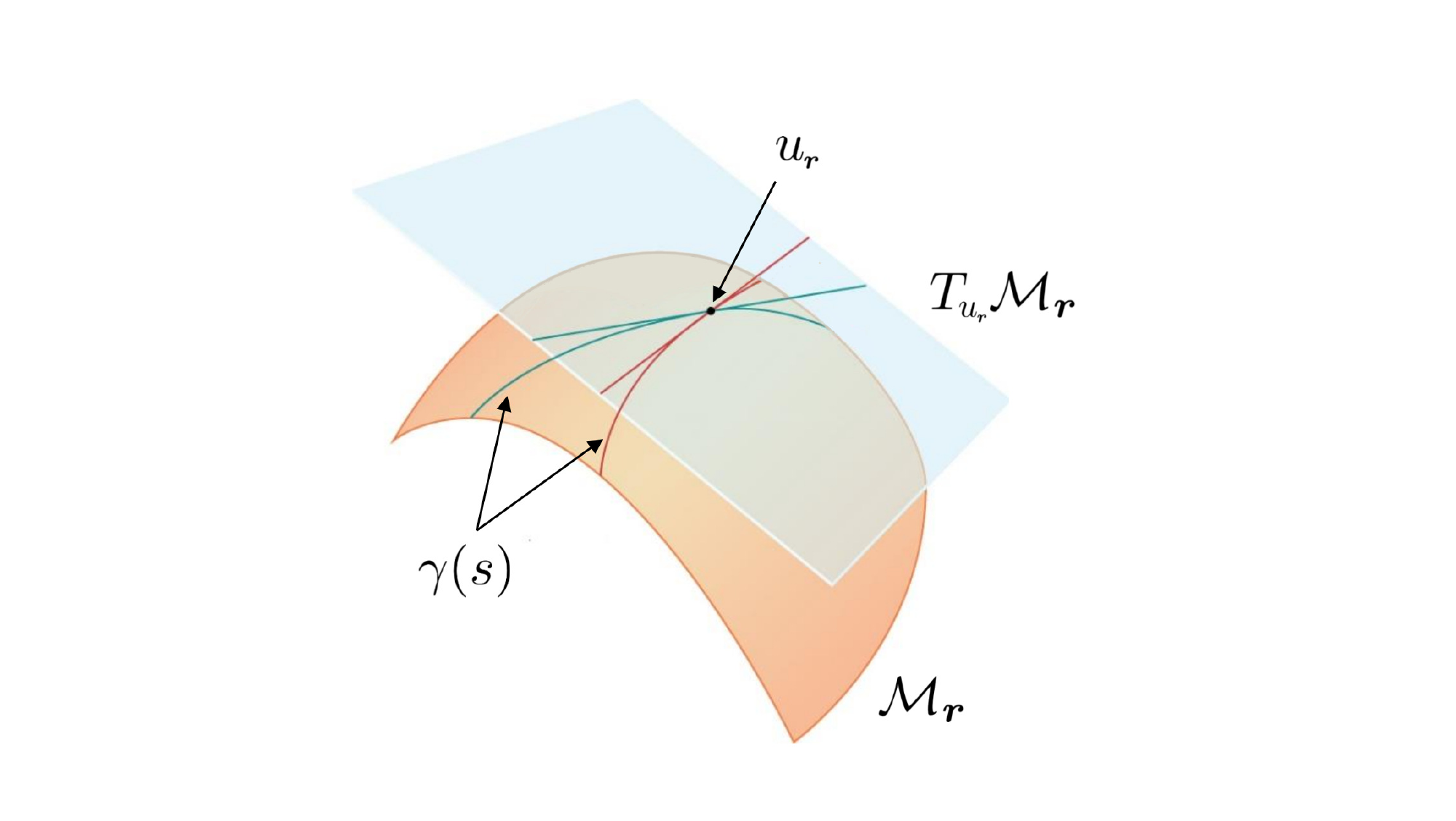}}
\caption{Sketch of the tensor manifold $\mathcal{M}_{\bm r}$ and
the tangent space $T_{{u_{\bm r}}}\mathcal{M}_{\bm r}$ at 
${u_{\bm r}} \in\mathcal{M}_{\bm r}$. The tangent space is defined 
as equivalence classes of velocities of continuously differentiable 
curves $\gamma(s)$ on $\mathcal{M}_{\bm r}$ 
passing through ${u_{\bm r}}$.}
\label{fig:tensor_manifold}
\end{figure}
\begin{equation}
\label{tangent_space}
T_{{u_{\bm r}}}\mathcal{M}_{\bm r} = \left\{\gamma'(s)\vert_{s=0}: 
\quad \gamma \in \mathcal{C}^1\left( (-\delta,\delta) , 
\mathcal{M}_{\bm r} \right), \quad  
\gamma(0) = {u_{\bm r}} \right\}.
\end{equation}
A sketch of $\mathcal{M}_{\bm r}$ and $T_{{u_{\bm r}}}\mathcal{M}_{\bm r}$ 
is provided in Figure \ref{fig:tensor_manifold}.
Since $L^2_{\mu}(\Omega)$ is an inner product space, for each 
$u \in L^2_{\mu}(\Omega)$ the tangent space 
$T_u L^2_{\mu}(\Omega)$ is canonically isomorphic 
to $L^2_{\mu}(\Omega)$. Moreover, for 
each ${u_{\bm r}} \in \mathcal{M}_{\bm r}$ the normal space 
to $\mathcal{M}_{\bm r}$ at the point ${u_{\bm r}}$, denoted by 
${N}_{{u_{\bm r}}} \mathcal{M}_{\bm r}$, 
consists of all vectors in $L^2_{\mu}(\Omega)$ that are 
orthogonal to $T_{{u_{\bm r}}} \mathcal{M}_{\bm r}$ 
with respect to the inner product in $L^2_{\mu}(\Omega)$
\begin{equation}
\label{normal_space}
N_{{u_{\bm r}}} \mathcal{M}_{\bm r} = \{w \in L^2_{\mu}(\Omega) : 
\left\langle w, v \right\rangle_{L^2_{\mu}(\Omega)} = 0, 
\quad \forall v \in T_{{u_{\bm r}}} \mathcal{M}_{\bm r}  \}.
\end{equation}
Since the tangent space $T_{{u_{\bm r}}} \mathcal{M}_{\bm r}$ 
is closed, for each point 
${u_{\bm r}} \in \mathcal{M}_{\bm r}$ the space 
$L^2_{\mu}(\Omega)$ admits a decomposition
into tangential and normal components
\begin{equation}
\label{tangent_normal_partition}
L^2_{\mu}(\Omega) = {T_{{u_{\bm r}}} \mathcal{M}_{\bm r}} 
\oplus N_{{u_{\bm r}}} \mathcal{M}_{\bm r}.
\end{equation}

\section{Dynamic tensor approximation of nonlinear PDEs}
\label{sec:approx_on_manifolds}

The idea of dynamic tensor approximation is to project the 
time derivative of a low-rank tensor onto the tangent 
space of the corresponding low-rank tensor manifold 
at each time. Such a projection results in evolution 
equations on the low-rank tensor manifold, 
and can be used to solve initial/boundary value problem 
of the form \eqref{nonlinear-ibvp}.
This approximation technique is known in the 
quantum physics community as 
Dirac-Frenkel/Mclachlan variational 
principle \cite{Raab,Frenkel,McLachlan}. 
Dynamic approximation has been 
recently studied by Lubich {\em et al.}
\cite{Lubich_2007,Lubich_2008,Lubich_2010,Lubich_2015} 
for finite-dimensional rank-structured manifolds 
embedded in Euclidean spaces.
There have also been extensions to the Tucker format 
on tensor Banach spaces \cite{Falco_2019} and 
tree-based tensor formats on 
tensor Banach spaces \cite{Falco_2020}.

\subsection{Dynamic tensor approximation on low-rank FTT manifolds}
\label{sec:dynamic_approximation} 

Let us briefly describe the method of dynamic tensor approximation 
for the low-rank FTT manifold \eqref{FTT_manifold}. 
First we define a projection onto 
the tangent space of $\mathcal{M}_{\bm r}$ at ${u_{\bm r}}$ by 
\begin{equation}
\label{tangent_projector}
\begin{aligned}
P_{{u_{\bm r}}} : {L^2_{\mu}(\Omega)} &\to T_{{u_{\bm r}}}\mathcal{M}_{\bm r} \\
P_{{u_{\bm r}}} v &= \argmin_{v_{\bm r} \in T_{{u_{\bm r}}} \mathcal{M}_{\bm r}} \|v - v_{\bm r} \|_{{L^2_{\mu}(\Omega)}}.
\end{aligned}
\end{equation}
For fixed ${u_{\bm r}}$, the map $P_{{u_{\bm r}}}$ is linear and bounded.
Each $v \in {L^2_{\mu}(\Omega)}$ admits a unique representation as $v = v_t + v_n$ 
where $v_t \in T_{{u_{\bm r}}} \mathcal{M}_{\bm r}$ and 
$v_n \in N_{{u_{\bm r}}} \mathcal{M}_{\bm r}$ (see equation \eqref{tangent_normal_partition}). 
From this representation it is clear that 
$P_{{u_{\bm r}}}$ is an 
orthogonal projection onto the tangent space 
$T_{{u_{\bm r}}} \mathcal{M}_{\bm r}$. 
If the initial condition ${u_0(\bm x)}$ 
is on the manifold $\mathcal{M}_{\bm r}$, then 
the solution to the initial/boundary value problem 
\begin{align}
\begin{cases}
\displaystyle\frac{\partial 
{u_{\bm r}} }{\partial t} = 
P_{{u_{\bm r}}} 
G({u_{\bm r}}), \vspace{0.1cm} \\
u(\bm x,0) = u_0(\bm x),
\end{cases}
\label{ibvp_dyn_approx} 
\end{align}
remains on the manifold $\mathcal{M}_{\bm r}$ 
for all $t \geq 0$. {Here $G$ is the 
nonlinear operator on the right hand side of 
equation \eqref{nonlinear-ibvp}}. The solution to 
\eqref{ibvp_dyn_approx} is known as a dynamic 
approximation to the 
solution of \eqref{nonlinear-ibvp}. 
In the context of separable Hilbert spaces, 
the dynamic approximation problem 
\eqref{ibvp_dyn_approx} can be solved using dynamically 
orthogonal or bi-orthogonal constraints on tensor modes 
\cite{Dektor_2020,Dektor_dyn_approx}. Such constraints, 
also referred to as gauge conditions, provide the 
unique solution of the minimization 
problem \eqref{tangent_projector} with 
different FTT cores. 
However, in the presence of repeated eigenvalues 
the bi-orthogonal constraints result in singular equations 
for the tangent space projection \eqref{tangent_projector}. 
Hereafter we recall the equations which allow us to 
compute  \eqref{tangent_projector} with FTT cores subject 
to dynamically orthogonal (DO) constraints. 

First, expand ${u_{\bm r}} \in \mathcal{M}_{\bm r}$ in terms of 
FTT cores ${u_{\bm r}} = \bm \Psi_1 \bm \Psi_2 \cdots \bm \Psi_d$, where 
$\bm \Psi_i$ are orthogonalized from the left, i.e., 
$\left\langle \bm \Psi_i^{\top} \bm \Psi_i \right\rangle_i 
= \bm I_{r_i \times r_i}$, for all $i = 1,\ldots,d-1$. 
With this ansatz, an arbitrary element of the tangent 
space $T_{{u_{\bm r}}}\mathcal{M}_{\bm r}$ can be expressed as 
\begin{equation}
\label{tangent_vector}
{\dot{u}_{\bm r}} = \dot{\bm \Psi}_1 \bm \Psi_{\geq 2} + \cdots + 
\bm \Psi_{\leq i-1} \dot{\bm \Psi}_i \bm \Psi_{\geq i+1} + 
\cdots + \bm \Psi_{\leq d-1} \dot{\bm \Psi}_d , 
\end{equation}
where ${\dot{u}_{\bm r}}=\partial {u_{\bm r}}/\partial t$ and 
$\dot{\bm \Psi}_i=\partial \bm \Psi_i/\partial t$. 
The DO constraints are given by 
\begin{equation}
\label{DO_constraints_multilevel}
\left\langle \dot{\bm \Psi}_i^{\top} \bm \Psi_i \right\rangle_i = 
\bm 0_{r_i \times r_i}, \qquad i = 1,\ldots,d-1,
\end{equation} 
which ensures that $\left\langle \bm \Psi_i^{\top}(t) \bm \Psi_i(t) \right\rangle_i = \bm I_{r_i \times r_i}$ for all $i = 1,\ldots,d-1$ 
and for all $t \geq 0$.
We have shown in \cite{Dektor_dyn_approx} that under 
these constraints, the convex minimization problem 
\eqref{tangent_projector} admits a unique minimum for 
vectors in the tangent space \eqref{tangent_vector} 
satisfying the PDE system
\begin{equation}
\label{DO-FTT_system}
\begin{aligned}
\dot{\bm \Psi}_1 &= \left[ \left\langle G({u_{\bm r}}) \bm 
\Psi_{\geq 2}^{\top} \right\rangle_{\geq 2} 
- \bm \Psi_1 \left\langle \bm \Psi_1^{\top} G({u_{\bm r}}) 
\bm \Psi_{\geq 2}^{\top} \right\rangle_{\geq 1} \right] 
\left\langle \bm \Psi_{\geq 2} \bm \Psi_{\geq 2}^{\top} 
\right\rangle_{ \geq 2}^{-1}, \\
\dot{\bm \Psi}_k &= \left[ \left\langle 
\bm \Psi_{\leq k-1}^{\top} G({u_{\bm r}})  
\bm \Psi_{\geq k+1}^{\top} 
\right\rangle_{\leq k-1,\geq k+1} - \right.\\
&\quad\quad \left. \bm \Psi_k \left\langle 
\bm \Psi_{\leq k}^{\top} G({u_{\bm r}})  
\bm \Psi_{\geq k+1}^{\top} \right\rangle_{\geq 1} \right] 
\left\langle \bm \Psi_{\geq k+1} \bm \Psi_{\geq k+1}^{\top} 
\right\rangle_{\geq k+1}^{-1}, \quad \quad  
k = 2,3,\ldots,d-1, \\
\dot{\bm \Psi}_d &= \left\langle 
\bm \Psi_{\leq d-1}^{\top} G({u_{\bm r}}) \right\rangle_{\leq d-1}.
\end{aligned}
\end{equation}
Here, ${u_{\bm r}}=\bm \Psi_1 \bm \Psi_2 \cdots \bm \Psi_d\in \mathcal{M}_r$
{ and we have 
introduced the notation 
\begin{equation}
\begin{aligned}
\langle \bm \Psi \rangle_{\leq k} &= \int_{\Omega_1 \times \cdots \times \Omega_k} \bm \Psi(\bm x) d \mu_1(x_1) \cdots \mu_k(x_k) , \\
\langle \bm \Psi \rangle_{\geq k} &= \int_{\Omega_k \times \cdots \times \Omega_d} \bm \Psi(\bm x) d \mu_k(x_k) \cdots \mu_d(x_d) , \\
\langle \bm \Psi \rangle_{\leq k-1, \geq k+1} &= \int_{\Omega_1 \times \cdots \times \Omega_{k-1} \times \Omega_{k+1} \times \cdots \times \Omega_d } 
\bm \Psi(\bm x) d \mu_1(x_1) \cdots \mu_{k-1}(x_{k-1}) \mu_{k+1}(x_{k+1}) \cdots \mu_d(x_d) ,
\end{aligned}
\end{equation}
for any matrix $\bm \Psi(\bm x) \in M_{r \times s}\left(L^2_{\mu}\left(\Omega\right)\right)$.
}
The DO-FTT system \eqref{DO-FTT_system} 
involves several inverse covariance matrices 
$\left\langle \bm \Psi_{\geq k} \bm \Psi_{\geq k}^{\top} 
\right\rangle_{\geq k}^{-1}$, which can become poorly 
conditioned in the presence of tensor modes 
with small energy (i.e. autocovariance matrices with small singular values). This phenomenon 
has been shown to be a result of the fact that the curvature of the 
tensor manifold at a tensor is inversely proportional to the smallest 
singular value present in the tensor 
\cite[section 4]{Lubich_2007}.
%
A slight improvement to the numerical stability 
of \eqref{DO-FTT_system} can be obtained by right 
orthogonalizing the partial products 
\begin{equation}
\label{right_orthogonalization_DO_FTT}
\bm \Psi_{\geq k} = \bm R_{k}^{\top} \bm Q_{\geq k}^{\top} , 
\qquad k = 2,\ldots, d.
\end{equation}
Using the orthogonality of $\bm Q_k$ it can easily be verified 
that $\bm R_k = \left\langle \bm \Psi_{\geq k} \bm \Psi_{\geq k}^{\top} \right\rangle_{\geq k}^{1/2}$.
With these right orthogonalized cores, 
the DO-FTT system \eqref{DO-FTT_system} can be written as 
\begin{equation}
\label{DO-FTT_system2}
\begin{aligned}
\dot{\bm \Psi}_1 &= \left[ \left\langle G({u_{\bm r}}) 
\bm Q_{\geq 2} \right\rangle_{\geq 2} 
- \bm \Psi_1 \left\langle \bm \Psi_1^{\top} G({u_{\bm r}}) 
\bm Q_{\geq 2} \right\rangle_{\geq 1} \right] \left\langle 
\bm \Psi_{\geq 2} \bm \Psi_{\geq 2}^{\top} 
\right\rangle_{\geq 2}^{-1/2}, \\
\dot{\bm \Psi}_k &= \left[ \left\langle 
\bm \Psi_{\leq k-1}^{\top} G({u_{\bm r}})  
\bm Q_{\geq k+1} \right\rangle_{\leq k-1,\geq k+1} - \right.\\
&\quad\quad \left. \bm \Psi_k \left\langle  \bm \Psi_{\leq k}^{\top} 
G({u_{\bm r}}) \bm Q_{\geq k+1} \right\rangle_{\geq 1} \right] 
\left\langle \bm \Psi_{\geq k+1} 
\bm \Psi_{\geq k+1}^{\top} \right\rangle_{\geq k+1}^{-1/2}, 
\quad \quad  k = 2,3,\ldots,d-1, \\
\dot{\bm \Psi}_d &= \left\langle \bm \Psi_{\leq d-1}^{\top} 
G({u_{\bm r}}) \right\rangle_{\leq d-1},
\end{aligned}
\end{equation}
where $\left\langle \bm \Psi_{\geq k} \bm \Psi_{\geq k}^{\top} 
\right\rangle_{k,\ldots,d}^{-1/2}$ denotes the inverse of the 
matrix square root. Since the condition number of  
$\left\langle \bm \Psi_{\geq k} \bm \Psi_{\geq k}^{\top} 
\right\rangle_{\geq k}$ is larger than the 
condition number of $\left\langle \bm \Psi_{\geq k} \bm \Psi_{\geq k}^{\top} \right\rangle_{\geq k}^{1/2}$, we have that 
the inverse covariances at the right hand side of 
\eqref{DO-FTT_system2} can be computed more 
accurately than the ones in \eqref{DO-FTT_system} 
in the presence of small singular values.

\subsubsection{Temporal integration using operator splitting methods}
\label{sec:splitting_integrator}
As we mentioned previously, 
one of the challenges of dynamic approximation of PDEs on 
low-rank tensor manifolds relates to the curvature of the manifold, 
which is proportional to the inverse of the smallest singular value 
of $\left\langle \bm \Psi_{\geq k} \bm \Psi_{\geq k}^{\top} 
\right\rangle_{\geq k}$ \cite[section 4]{Lubich_2007}. 
Such curvature appears naturally at the right hand side of the 
DO-FTT system \eqref{DO-FTT_system} in the form of 
inverse covariances  $\left\langle \bm \Psi_{\geq k} 
\bm \Psi_{\geq k}^{\top} \right\rangle^{-1}_{\geq k}$.
Clearly, if the tensor solution is comprised of 
cores with small singular values, then the covariance matrices 
$\left\langle \bm \Psi_{\geq k} \bm \Psi_{\geq k}^{\top} 
\right\rangle_{\geq k}$ are ill-conditioned and 
therefore not easily invertible. Moreover, it is 
desirable to add and remove tensor modes adaptively 
during temporal integration, and adding a mode 
with zero energy immediately yields singular 
covariance matrices (see \cite{Dektor_2020}). 
The problem of inverting the covariance matrices 
$\left\langle \bm \Psi_{\geq k} \bm \Psi_{\geq k}^{\top} 
\right\rangle_{\geq k}$ when integrating \eqref{DO-FTT_system} 
or \eqref{DO-FTT_system2} can be avoided by using projector-splitting
methods. These methods were originally proposed for 
integration on tensor manifolds by 
Lubich {\em et. al} in \cite{Lubich_2015,Lubich_2016,Lubich_2018}. 
The key idea is to apply an exponential operator 
splitting scheme, e.g., the Lie-Trotter scheme, directly to the 
projection operator onto the tangent space defining the 
dynamic approximation (see equation \eqref{tangent_projector}). 
To describe the method, we begin by introducing a general 
framework for operator splitting of dynamics on the FTT 
tangent space. We first rewrite the right hand side of \eqref{ibvp_dyn_approx} as 
\begin{equation}
\label{projected_RHS}
\begin{aligned}
    P_{{u_{\bm r}}} G({u_{\bm r}}) &= \dot{\bm \Psi}_1 \bm \Psi_{\geq 2} + 
   \bm \Psi_1 \dot{\bm \Psi}_2 \bm \Psi_{\geq 3} + 
   \cdots + \bm \Psi_{\leq d-1} \dot{\bm \Psi}_d \\
   &= \dot{\bm \Psi}_1 \left\langle \bm \Psi_{\geq 2} \bm \Psi_{\geq 2}^{\top} \right\rangle_{\geq 2}^{1/2} \bm Q_{\geq 2}^{\top} + 
   \bm \Psi_1 \dot{\bm \Psi}_2 \left\langle \bm \Psi_{\geq 3} \bm \Psi_{\geq 3}^{\top} \right\rangle_{\geq 3}^{1/2} \bm Q_{\geq 3}^{\top} + 
   \cdots + \bm \Psi_{\leq d-1} \dot{\bm \Psi}_d,
\end{aligned}
\end{equation}
where in the second line we used the right orthogonalizations 
in equation \eqref{right_orthogonalization_DO_FTT}.
A substitution of the expressions for $\dot{\bm \Psi}_k$ 
we obtained in \eqref{DO-FTT_system2} into 
\eqref{projected_RHS} yields 
\begin{equation}
\label{proj_op_for_splitting}
    P_{{u_{\bm r}}} G({u_{\bm r}}) = P_d^+ G({u_{\bm r}}) + \sum_{i=1}^{d-1} P_i^+ G({u_{\bm r}})  - P_i^- G({u_{\bm r}}) ,
\end{equation}
where we defined the following projection operators 
from $L^2_{\mu}(\Omega)$ onto $T_{{u_{\bm r}}} \mathcal{M}_{\bm r}$
\begin{equation}
\label{projections}
    \begin{aligned}
    P_k^{+} z(\bm x) &= \bm \Psi_{ \leq k-1} \left\langle \bm \Psi_{ \leq k-1}^{\top} z(\bm x)  \bm Q_{\geq k+1} \right\rangle_{\leq k-1, \geq k+1} \bm Q_{\geq k+1}^{\top} , \qquad k = 1,\ldots, d, \\
    P_k^{-} z(\bm x) &= \bm \Psi_{\leq k} \left\langle \bm \Psi_{\leq k}^{\top} z(\bm x) \bm Q_{\geq k+1} \right\rangle_{\geq 1} \bm Q_{\geq k+1}^{\top} , \qquad k = 1,\ldots, d-1,
    \end{aligned}
\end{equation}
{for any $z(\bm x) \in L^2_{\mu}(\Omega)$}. Also we set $\bm \Psi_0 = 1$.
The key point in \eqref{proj_op_for_splitting} is that inverse 
covariance matrices no longer appear.
To establish a general operator splitting framework, 
let us assume that there exists an evolution operator 
$\mathcal{E}_{P_{{u_{\bm r}}} G}$ for the solution of the initial/boundary 
value problem \eqref{ibvp_dyn_approx}, where $P_{{u_{\bm r}}} G$ 
is given in \eqref{proj_op_for_splitting}. Such an evolution 
operator $\mathcal{E}_{P_{{u_{\bm r}}} G}: L^2_{\mu}(\Omega) \times 
[0,T] \to L^2_{\mu}(\Omega)$ satisfies a 
semi-group property and it maps the 
initial condition $u_0(\bm x)$ into the solution 
to \eqref{ibvp_dyn_approx} at a later time
\begin{equation} 
u(\bm x,t) = \mathcal{E}_{P_{{u_{\bm r}}} G}(u_0(\bm x),t).
\end{equation}
We write such an evolution operator formally as an 
exponential operator with generator $D_{P_{{u_{\bm r}}} G}$ 
(see e.g. \cite{Koch_2013})
\begin{equation}
u(\bm x,t) = e^{t D_{P_{{u_{\bm r}}} G}} u_0(\bm x), \qquad 0 \leq t \leq T,
\end{equation}
where $D_{P_{{u_{\bm r}}} G}$ is the Lie derivative associated with $P_{{u_{\bm r}}} G$.
We now discretize the temporal domain of interest  
$[0, T]$ into $N+1$ evenly-spaced time instants,
\begin{equation}
\label{discrete_time}
t_i = i \Delta t, \qquad 
\Delta t = \frac{T}{N}, \qquad i=0,1,\ldots,N.
\end{equation}
An approximation to the exact solution 
of \eqref{ibvp_dyn_approx} is then 
obtained by the recurrence relation 
{
\begin{equation}
    u_{\bm r}(\bm x,t_{n+1}) \approx \mathcal{S}(\Delta t, u_{\bm r}(\bm x,t_{n})),
\end{equation}}
where $\mathcal{S}$ is an exponential operator splitting 
that approximates the exact evolution operator 
\begin{equation}
\label{general_splitting_operator}
    \mathcal{S}(t,\cdot) = \prod_{i=1}^s \left[ e^{ \gamma_{i,d}t P_d^+ G } \prod_{j=1}^{d-1} \left( e^{\gamma_{i,j} t P_j^+ G } e^{\gamma_{i,j} t P_j^- G } \right) \right].
\end{equation}
Setting $s = 1$ and $\gamma_{1,j} = 1$ for all $ j = 1,\ldots,d$ in 
\eqref{general_splitting_operator} yields the well-known 
Lie-Trotter splitting, which is first-order in time.
{
The discrete time version of this scheme can be written as
\begin{equation}
\label{split_diff_eqs}
\left\{
    \begin{aligned}
    u_1^+(t_{i+1}) &= u_1^+(t_i) + \Delta t P_1^+ G(u_{\bm r}),  \qquad u_1^+(t_i) = u_{\bm r}(t_i), \\
    u_1^-(t_{i+1}) &= u_1^-(t_{i}) - \Delta t P_1^-G(u_{\bm r}), \qquad u_1^-(t_i) = u_1^+(t_{i+1}), \\
    \vdots \\
    u_j^+(t_{i+1}) &= u_j^+(t_{i}) + \Delta t P_j^+ G(u_{\bm r}),  \qquad u_j^+(t_i) = u_{j-1}^-(t_{i+1}), \\
    u_j^-(t_{i+1}) &= u_j^-(t_{i}) - \Delta t P_j^- G(u_{\bm r}), \qquad u_j^-(t_i) = u_j^+(t_{i+1}), \\
    &\vdots \\
    u_d^+(t_{i+1}) &= u_d^+(t_{i}) + \Delta t P_d^+ G(u_{\bm r}), \qquad u_d(t_i) = u_{d-1}^-(t_{i+1}), \\
    u_{\bm r}(t_{i+1}) &= u_d^+(t_{i+1}).
    \end{aligned}\right.
\end{equation}
This allows us to compute $u_{\bm r}(t_{i+1})$ given $u_{\bm r}(t_i)$. 
Although each equation in \eqref{split_diff_eqs} involves a FTT tensor, it was shown in 
\cite[Theorem 4.1]{Lubich_2015} that each equation only updates one tensor core. Clearly this is 
computationally more efficient than updating a full tensor. 
Moreover, in \eqref{split_diff_eqs} there is no need to invert covariance matrices, 
which is a distinct advantage over iterating a discrete form of \eqref{DO-FTT_system} or \eqref{DO-FTT_system2}. }

{
Regarding computational cost, suppose we discretize the $d$-dimensional domain $\Omega$ using a 
tensor product grid with $n$ points per dimension. 
It was pointed out in \cite{Lubich_2015} that the computational complexity of the sweeping algorithm to 
update the tensor cores for the Lie-Trotter scheme \eqref{split_diff_eqs} applied to a linear PDE 
(i.e. equation \eqref{ibvp_dyn_approx} with linear $G$) is linear in the dimension $d$ 
but has high polynomial complexity in the tensor rank. 
On the other hand, discretizing such linear PDE on the same tensor product grid and performing one time step 
with a first-order time stepping scheme (e.g. Euler forward) has computational complexity which scales exponentially with the dimension $d$. 
Specifically, assuming that the operator $G$ in \eqref{nonlinear-ibvp} is linear with rank $r_{G}$ (see \cite{Beylkin}), the computational 
cost of one time step of Euler forward is $d n^{d+1} r_{G} + n^d r_{G}$ floating point operations, hence exponential in $d$.
}

\subsection{Step-truncation temporal integration methods}
\label{sec:step_truncation}
Another methodology to integrate nonlinear PDEs  
on fixed-rank tensor manifolds $\mathcal{M}_{\bm r}$ 
is step-truncation \cite{Vandereycken_2019,
rodgers2020step-truncation, rodgers2020stability}. 
The idea is to integrate the solution off of $\mathcal{M}_{\bm r}$ for 
short time, e.g., by performing one time step of the full equation 
with a conventional time-stepping scheme, followed by a 
truncation operation back onto $\mathcal{M}_{\bm r}$.
To describe this method further let us define the truncation operator 
\begin{equation}
\begin{aligned}
\mathfrak{T}_{\bm r} : {L^2_{\mu}(\Omega)} &\to \mathcal{M}_{\bm r} \\
\mathfrak{T}_{\bm r}(u) &= \argmin_{u_{\bm r} \in \mathcal{M}_{\bm r}} \| u - u_{\bm r} \|_{{L^2_{\mu}(\Omega)}}, 
\end{aligned}
\end{equation}
which provides the best approximation of $u$ on $\mathcal{M}_{\bm r}$. 
Such a map is known as a metric projection or closest point function 
and in general it may be multivalued, i.e., the set of $u_{\bm r} \in \mathcal{M}_{\bm r}$ 
which minimize $\| u - u_{\bm r} \|_{{L^2_{\mu}(\Omega)}}$ is not a singleton set. 
However, since $\mathcal{M}_{\bm r}$ 
is a smooth submanifold of ${L^2_{\mu}(\Omega)}$, we have by \cite[Proposition 5.1]{Salas_2019} 
that for each $u_0 \in \mathcal{M}_{\bm r}$ there exists an open neighborhood $U$ of 
$u_0$ such that $\mathfrak{T}_{\bm r}$ is well-defined and 
smooth on $U$.
Let 
\begin{equation}
\label{one_step_scheme}
    u(\bm x,t_{k+1}) = u(\bm x,t_{k}) + \Delta t \Phi\left(G,u(\bm x,t_k), \Delta t\right)
\end{equation} 
be a convergent one-step time integration 
scheme\footnote{Time stepping schemes of the form 
\eqref{one_step_scheme} include Runge-Kutta methods and linear 
multi-step methods \cite{rodgers2020step-truncation}.} approximating the solution to the initial 
value problem \eqref{nonlinear-ibvp}. 
Assume that the solution $u(\bm x,t_0)$ at time $t_0$ 
is on $\mathcal{M}_{\bm r}$.\footnote{ 
If $u(\bm x,t_0)$ is not on $\mathcal{M}_{\bm r}$ then it may be mapped onto $\mathcal{M}_{\bm r}$ 
by evaluating $\mathfrak{T}_{\bm r}(u(\bm x,t_0))$.} In order to guarantee the solution 
$u(\bm x,t_k)$ at time step $t_k$ is an element of the manifold $\mathcal{M}_{\bm r}$ 
for each $k = 1,2,\ldots$, we apply the truncation 
operator to the right hand 
side \eqref{one_step_scheme}. 
This yields the following step-truncation method 
\begin{equation}
\label{step_truncation_scheme}
u_{\bm r}(\bm x,t_{k+1}) = \mathfrak{T}_{\bm r} 
\left( u_{\bm r}(\bm x,t_k) + 
\Delta t \Phi\left(G,u_{\bm r}(\bm x,t_k), \Delta t\right) \right).
\end{equation}

\subsection{Consistency of dynamic approximation and step-truncation methods}

Next we ask what happens in the step-truncation 
algorithm in the limit of time step 
$\Delta t$ approaching zero. The result of such 
a limiting procedure results in a scheme 
which keeps the solution $u(\bm x,t)$ 
on the manifold $\mathcal{M}_{\bm r}$ for all time 
$t \geq t_0$ in an optimal way. 
We now show that this limiting procedure in fact 
results in precisely the dynamic approximation method 
described in section \ref{sec:dynamic_approximation}.
In other words, by sending $\Delta t$ to zero in \eqref{step_truncation_scheme} we obtain a 
solution of \eqref{ibvp_dyn_approx}. For similar 
discussions connecting these two approximation 
methods in closely related contexts see 
\cite{Lermusiaux_2018,Lermusiaux_2019,Vandereycken_2019}.
To prove consistency between step-truncation 
and dynamic approximation methods we need to compute 
$\mathfrak{T}_{\bm r}(u(\bm x,t))$ for $t$ infinitesimally 
close to $t_0$. Such a quantity depends on the derivative  
\begin{equation}
    \left.\frac{\partial \mathfrak{T}_{\bm r}(u(\bm x,t))}
    {\partial t}\right|_{t = t_0}
    = \lim_{\Delta t \to 0 } \frac{\mathfrak{T}_{\bm r}(u(\bm x,t)) -
     \mathfrak{T}_{\bm r}(u(\bm x,t_0))}{\Delta t}.
\end{equation}
The following proposition provides a representation of 
the derivative $\partial \mathfrak{T}_{\bm r}(u(\bm x,t))/\partial t$
in terms of $G(u(\bm x,t))$ and the Fr\'echet derivative \cite{VenturiSpectral} of the operator $\mathfrak{T}_{\bm r}(u)$.

\begin{proposition}
\label{prop:derivative}
If the solution $u_0 = u(\bm x,t_0)$ 
to \eqref{nonlinear-ibvp} at time 
$t_0$ is on the manifold $\mathcal{M}_{\bm r}$, then 
\begin{equation}
\frac{\partial \mathfrak{T}_{\bm r}(u(\bm x,t))}{\partial t}\biggr\vert_{t = t_0} = (\mathfrak{T}_{\bm r})'_{u_0} G(u(\bm x,t)),
\label{dTdt}
\end{equation}
where $(\mathfrak{T}_{\bm r})'_{u_0}$ is the 
Fr\'echet derivative of the nonlinear operator 
$\mathfrak{T}_{\bm r}$ at the point $u_0$.
\end{proposition}
\begin{proof}
Express the solution of \eqref{nonlinear-ibvp} 
at time $t \geq t_0$ as 
\begin{equation}
u(\bm x,t) = u_0(\bm x) + h(\bm x,t),
\end{equation}
where 
\begin{equation}
h(\bm x,t) = \int_{t_0}^t G(u(\bm x,\tau)) d\tau.
\end{equation}
Expanding $\mathfrak{T}_{\bm r}(u(\bm x,t))$ in a Taylor series around $u_0(\bm x)$ 
we obtain \cite[Theorem 6.1]{Nashed}
\begin{equation}
\label{truncation_taylor}
\begin{aligned}
\mathfrak{T}_{\bm r}(u(\bm x,t)) &= u_0(\bm x) + (\mathfrak{T}_{\bm r})'_{u_0} h(\bm x,t) + \frac{1}{2}(\mathfrak{T}_{\bm r})''_{u_0} h(\bm x,t)^2 + \cdots .
\end{aligned}
\end{equation}
Differentiating \eqref{truncation_taylor} with respect to $t$ and 
evaluating at $t = t_0$ we obtain
\begin{equation}
\label{truncation_derivative}
\frac{\partial \mathfrak{T}_{\bm r}(u(\bm x,t))}{\partial t}\biggr\vert_{t = t_0} = (\mathfrak{T}_{\bm r})'_{u_0} G(u(\bm x,t_0)),
\end{equation}
where we assumed that $\partial/\partial t$ commutes 
with $(\mathfrak{T}_{\bm r})'_{u_0}$ and used the fact that 
$\partial h(\bm x,t)/\partial t = G(u(\bm x,t))$ for the first order term. 
All of the higher order terms are seen to be zero by commuting $\partial/\partial t$ 
with $(\mathfrak{T}_{\bm r})^{(n)}_{u_0}$ and using chain rule.

\end{proof}

\vs
\noindent
Since $\mathfrak{T}_{\bm r}(u(\bm x,t))$ is an element of $\mathcal{M}_{\bm r}$ for all $t \geq t_0$, 
it follows that \eqref{dTdt} is an element of $T_{u_0}\mathcal{M}_{\bm r}$. 
Arguing on the optimality of the tangent space 
element $(\mathfrak{T}_{\bm r})'_{u_0} G(u(\bm x,t_0))$ it is seen that  \eqref{truncation_derivative} is 
the same problem as dynamic approximation \eqref{ibvp_dyn_approx}, i.e., 
$(\mathfrak{T}_{\bm r})'_{u_0} = P_{u_0}$. 
Now consider the scheme \eqref{step_truncation_scheme} and use a Taylor 
expansion of $\mathfrak{T}_{\bm r}$ around $u_{\bm r}(\bm x,t_k)$ on the right hand side 
\begin{equation}
    \label{step_truncation_scheme_taylor_expanded}
    u_{\bm r}(\bm x,t_{k+1}) =   u_{\bm r}(\bm x,t_k) + \Delta t (\mathfrak{T}_{\bm r})' \Phi\left(G,u_{\bm r}(\bm x,t_k), \Delta t\right) + O(\Delta t^2).
\end{equation}
Discarding higher order terms in $\Delta t$ yields
\begin{equation}
    \label{step_truncation_to_dyn_approx}
    u_{\bm r}(\bm x,t_{k+1}) \simeq   u_{\bm r}(\bm x,t_k) + \Delta t P_{u_{\bm r}} \Phi\left(G,u_{\bm r}(\bm x,t_k), \Delta t\right) .
\end{equation}
Moreover if the increment function $\Phi$ defines 
the Euler forward scheme 
\begin{equation} 
\Phi(G,u_{\bm r}(\bm x,t_k),\Delta t) = G(u_{\bm r}(\bm x,t_k)),
\end{equation}
then the scheme \eqref{step_truncation_to_dyn_approx} is equivalent to the 
scheme in \eqref{ibvp_dyn_approx}. 
Thus, we just proved the following lemma.  
\begin{lemma}
Step-truncation and dynamic approximation methods 
are consistent at least to first-order in $\Delta t$.
\label{lemma:consistency}
\end{lemma}
\noindent
This Lemma applies to any first-order time integrator for dynamic approximation 
and step-truncation, including 
the Lie-Trotter splitting integrator we discussed in section 
\ref{sec:splitting_integrator}. 

\section{Rank-adaptive integration}
\label{sec:adaptive_rank}

The solution to the initial/boundary value problem 
\eqref{nonlinear-ibvp} is often not accurately represented 
on a tensor manifold with fixed rank, even for short 
integration times.  In this section we discuss effective 
methods to adaptively add and remove tensor 
modes from the solution based on appropriate criteria.

In the context of step-truncation algorithms, if 
the solution rank naturally decreases in time 
then the operator $\mathfrak{T}_{\bm r}$ in
 \eqref{step_truncation_scheme} is no longer 
well-defined. In this situation, replacing the operator 
$\mathfrak{T}_{\bm r}$ with $\mathfrak{T}_{\bm s}$ 
for an appropriate\footnote{Here $\leq$ denotes component-wise 
inequality of rank vectors, i.e., $\bm s<\bm r$ if 
and only if $s_i \leq r_i$ for all $i = 0,1,\ldots,d$.}  
$\bm s \leq \bm r$ allows for integration to continue.
On the other hand, if the solution rank increases in 
during integration then the operator $\mathfrak{T}_{\bm r}$ will still be 
well-defined for small enough $\Delta t$ but the 
approximation on $\mathcal{M}_{\bm r}$ will not 
retain accuracy. To address this problem of constant 
rank integration we shall introduce a criterion for 
rank increase of the FTT solution. 
Both decreasing and increasing rank are based on 
FTT {orthogonalization and} truncation (see section \ref{sec:orthogonalization}).
For the remainder of this section let $u(\bm x,t)$ be the 
solution to \eqref{nonlinear-ibvp} and 
$u_{\bm r}(\bm x,t) \in \mathcal{M}_{\bm r}$ an 
approximation of $u(\bm x,t)$ obtained by either the 
solution of the dynamical approximation problem 
\eqref{ibvp_dyn_approx} or step-truncation 
methods (see section \ref{sec:step_truncation}).

\subsection{Decreasing tensor rank}
\label{sec:decreasing_rank}
For decreasing tensor rank at time $t$, we are interested in 
determining if $u_{\bm r}(\bm x,t) \in \mathcal{M}_{\bm r}$ 
is close to an element $u_{\bm s}(\bm x,t) \in \mathcal{M}_{\bm s}$ 
for $\bm s \leq \bm r$. This can be achieved by simply 
performing a FTT truncation on $u_{\bm r}(\bm x,t)$ 
with small threshold $\epsilon_{\mathrm{dec}}$.
Since the splitting integrator described 
in section \ref{sec:splitting_integrator} 
is robust to over approximation by tensor rank, it may not be 
strictly necessary to decrease rank during integration. 
However, it is desirable to have solutions 
of the lowest rank possible (while retaining accuracy) 
when solving high dimensional problems. 
For these reasons it is advisable not perform a FTT truncation at each 
time step (as this would be unnecessary and inefficient when using an operator 
splitting integrator) but only every once and a while. 
One may choose a criterion for when to check for 
rank decrease based on the problem, step size, current rank, and dimension.
If one is using a step-truncation method with a tolerance based 
FTT truncation algorithm such as the one described in section \ref{sec:orthogonalization} 
then rank decrease is already built into each time step.

\subsection{Increasing tensor rank}
\label{sec:increase_rank}
As a general heuristic one would like to increase 
rank at the time when the error between the 
low-rank approximation $u_{\bm r}(\bm x,t)$ 
and the PDE solution $u(\bm x,t)$ will become 
large after the subsequent time step. 
Such critical time instant for rank increase 
can be determined by examining the normal component 
of the dynamics
\begin{equation}
N_{{u_{\bm r}}}(G({u_{\bm r}})) = G({u_{\bm r}}) - P_{{u_{\bm r}}}(G({u_{\bm r}})).
\end{equation}
To describe this situation further, 
suppose we are integrating one time step forward 
from $t_i$ to $t_{i+1}$. The error at $t_{i+1}$ is given by 
\begin{equation}
    \begin{aligned}
    E(t_i,t_{i+1}) &= u_{\bm r}(\bm x,t_{i+1}) - u(\bm x,t_{i+1}) \\
    &= u(\bm x,t_i) + \int_{t_i}^{t_{i+1}} G(u(\bm x,\tau)) d\tau - \left( \mathfrak{T}_{\bm r} ( u(\bm x,t_i) )
    + \int_{t_i}^{t_{i+1}} P_{u_{\bm r}(\bm x,\tau)} G(u_{\bm r}(\bm x,\tau)) d\tau \right).
    \end{aligned} 
\end{equation}
If $u(\bm x,t_i) \in \mathcal{M}_{\bm r}$ then 
\begin{equation}
\label{error_1}
    E(t_i,t_{i+1}) = \int_{t_i}^{t_{i+1}}\left[ G(u(\bm x,\tau)) - P_{u_{\bm r}(\bm x,\tau)} G(u_{\bm r}(\bm x,\tau)) \right]
    d\tau.
\end{equation}
For small $\Delta t$ the above integral can be approximated 
by the left endpoint 
\begin{equation}
\label{left_endpoint_approx_of_normal}
\begin{aligned}
    E(t_i,t_{i+1}) &=\Delta t \left( G(u_{\bm r}(\bm x,t_i)) - P_{u_{\bm r}(\bm x,t_i)} G(u_{\bm r}(\bm x,t_i)) \right) + O(\Delta t^2) \\
    &= \Delta t N_{u_{\bm r}(\bm x,t_i)}(G(u_{\bm r}(\bm x,t_i))) + O(\Delta t^2),
\end{aligned}
\end{equation}
where $N_{u_{\bm r}(\bm x,\tau)}$ denotes the 
orthogonal projection onto the normal space of 
$\mathcal{M}_{\bm r}$ at the point $u_{\bm r}(\bm x,t)$.
Hence, up to first-order in $\Delta t$ we have that 
\begin{equation}
    \|E(t_i,t_{i+1})\| \simeq 
    \Delta t \|N_{u_{\bm r}(\bm x,t_i)}(G(u_{\bm r}(\bm x,t_i)))\|.
\end{equation}
From this approximation we see that a reasonable criterion for increasing 
rank at time $t_i$ is when the norm of the normal component 
of $G(u_{\bm r}(\bm x,t_i))$ is larger than some threshold $\epsilon_{\mathrm{inc}}$ (see Figure \ref{fig:normal component})
\begin{equation}
\label{condition_for_rank_inc}
   \| N_{u_{\bm r}(\bm x,t_i)}(G(u_{\bm r}(\bm x,t_i)))\| > \epsilon_{\mathrm{inc}}.
\end{equation}
\begin{figure}\centerline{\includegraphics[height=6.5cm]{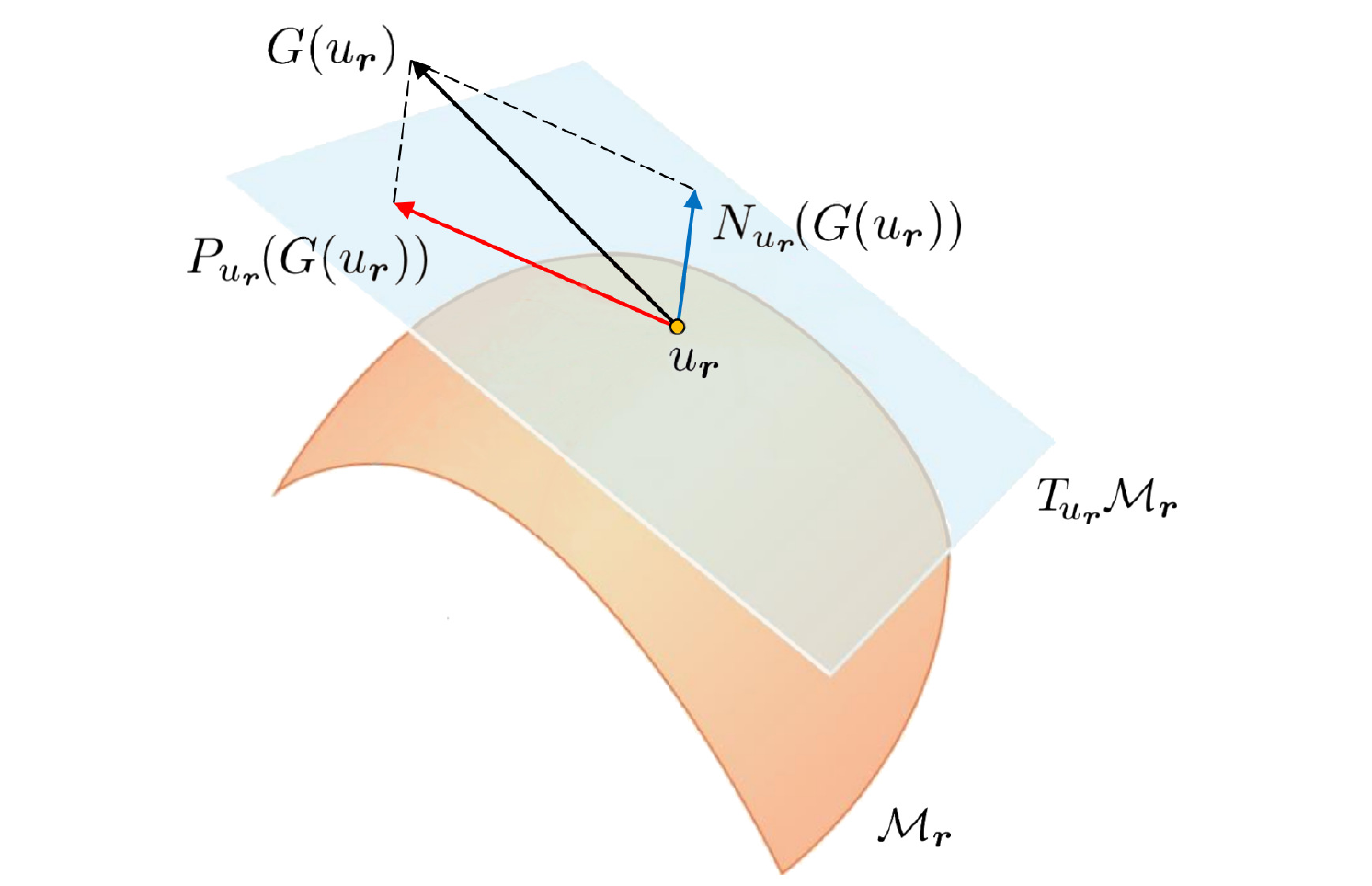}}
\caption{Tangent and normal components of $G\left(u_{\bm r}\right) 
= \partial u_{\bm r}/\partial t$ at $u_{\bm r}$.  The tensor rank 
of the solution is increased at time $t_i$ if the norm 
of the normal component $N_{u_{\bm r}}(G(u_{\bm r}))$ is 
larger than a specified threshold $\epsilon_{\mathrm{inc}}$.}
\label{fig:normal component}
\end{figure}
To efficiently compute the normal component 
$N_{u_{\bm r}(\bm x,t_i)}(G(u_{\bm r}(\bm x,t_i)))$ 
at each time instant $t_i$ we use the formula 
\begin{equation}
N_{u_{\bm r}(\bm x,t_i)}(G(u_{\bm r}(\bm x,t_i))) = G(u_{\bm r}(\bm x,t_i)) - P_{u_{\bm r}(\bm x,t_i)}(G(u_{\bm r}(\bm x,t_i))),
\end{equation}
where $N_{\bm r}(G(u_{\bm r}))$ and $T_{u_{\bm r}}(G(u_{\bm r}))$ represent the normal and tangential components of $G(u_{\bm r})$.
The tangential component can be approximated at a 
low computational cost via backward differentiation 
formulas (BDF) as  
\begin{align}
\label{backwards_fd_1st_order}
\widetilde{P}^{(2)}_{u_{\bm r}} G(u_{\bm r}) = &\frac{u_{\bm r}(\bm x, t_{i}) - u_{\bm r}(\bm x, t_{i-1})}{\Delta t} + O(\Delta t^2)  &\text{(two-point formula)},\\
\label{backwards_fd_2nd_order}
 \widetilde{P}^{(3)}_{u_{\bm r}} G(u_{\bm r}) = &\frac{3u_{\bm r}(\bm x, t_i) - 4u_{\bm r}(\bm x, t_{i-1}) + u_{\bm r}(\bm x,t_{i-2})}{2\Delta t} + O(\Delta t^3)
   &\text{(three-point formula)},\\
    \label{backwards_fd_p_order}
 \widetilde{P}^{(p)}_{u_{\bm r}} G(u_{\bm r}) = & BD_p(\Delta t, u_{\bm r}(\bm x,t_i), u_{\bm r}(\bm x,t_{i-1}), \ldots, u_{\bm r}(\bm x,t_{i-p})) + O(\Delta t^{p+1})   &\text{($p$-point formula)}.
\end{align}
With a $p$-point backward difference approximation 
of the tangent space projection available at $t_i$ we easily obtain 
an approximation of the normal component of $G(u_{\bm r})$ at $t_i$ 
\begin{equation}
\label{normal_component}
    N_{u_{\bm r}(\bm x,t_i)}(G(u_{\bm r}(\bm x,t_i))) = G(u_{\bm r}(\bm x,t_i)) - \widetilde{P}^{(p)}_{u_{\bm r}(\bm x,t_i)} G(u_{\bm r}(\bm x,t_i))  + O(\Delta t^{p+1}),
\end{equation}
which allows us to implement the 
criterion \eqref{condition_for_rank_inc} for rank 
increase at time $t_i$. Clearly, the $p$-point formula 
\eqref{backwards_fd_p_order}, and the corresponding 
approximation of the normal component \eqref{normal_component}, 
are effectively of order $p$ in $\Delta t$ if and only if 
the time snapshots $u_{\bm r}(\bm x,t_i)$ are computed 
via a temporal integrator of order $p$. 
We emphasize that this method of using a finite difference 
stencil based on the temporal grid for approximating 
the tangential component of the 
dynamics (and thus the normal component) 
creates a lower bound for the choice of normal vector 
threshold $\epsilon_{\mathrm{inc}}$. 
In particular, we must have that $K_1(\Delta t)^p \geq \epsilon_{\text{inc}}$ for some constant $K_1$ 
otherwise the error incurred from our approximation 
of the normal component may trigger unnecessary mode addition.
This approximation of the normal component is cheap but only 
informs on whether or not it is appropriate to add modes at 
time instant $t_i$. 

The subsequent question is which entries of the rank 
vector $\bm r$ need to be increased. In order to make such a 
determination we expand the approximate 
solution at time $t$ as 
\begin{equation}
\label{over_represented_solution}
u_{\bm r}(\bm x,t) = \bm \Psi_1(t) \cdots \bm \Psi_d(t) + 
\bm \Gamma_1(t) \cdots \bm \Gamma_d(t), \qquad 
\bm \Psi_i \in M_{r_{i-1} \times r_i}(L^2_{\mu_i}(\Omega_i)), 
\quad \bm \Gamma_i \in M_{f_{i-1} \times f_i}
(L^2_{\mu_i}(\Omega_i)),
\end{equation}
where $\bm \Gamma_1(t) \cdots \bm \Gamma_d(t) = 0$ 
for all $t \in [0,T]$. Differentiating 
\eqref{over_represented_solution} with respect to 
time yields
\begin{equation}
    \frac{\partial u_{\bm r}(\bm x,t)}{\partial t} = \frac{\partial}{\partial t} \left[ \bm \Psi_1(t) \cdots \bm \Psi_d(t) \right] + \frac{\partial }{\partial t} \left[ \bm \Gamma_1(t) \cdots \bm \Gamma_d(t) \right].
\end{equation}
Subtracting off the tangential 
component \eqref{tangent_vector} we have the normal 
component at time $t$ 
\begin{equation}
    \label{normal_comp_in_FTT_format}
    N_{u_{\bm r}(\bm x,t)}\left(\frac{\partial u_{\bm r}(\bm x,t)}{\partial t}\right) = \frac{\partial}{\partial t} \left[ \bm \Gamma_1(t) \cdots \bm \Gamma_d(t) \right].
\end{equation}
Next, orthogonalize the partial product $\bm \Gamma_{\leq i-1}(t)$ 
from the left and the partial product $\bm \Gamma_{\geq i}(t)$ 
from the right to obtain 
\begin{equation}
\label{time_deriv_of_normal}
    N_{u_{\bm r}(\bm x,t)}\left(\frac{\partial u_{\bm r}(\bm x,t)}{\partial t}\right) = \frac{\partial}{\partial t} 
    \left[ \bm \Gamma_1(t) \cdots \bm \Gamma_{i-1}(t) \bm C_i(t) \bm \Gamma_{i}^{\top}(t) \cdots \bm \Gamma_d^{\top}(t) \right],
\end{equation}
where $\bm C_i = \bm 0_{r_{i-1} \times r_i}$ and $\left\langle \bm \Gamma_i^{\top} \bm \Gamma_i \right\rangle_i = \bm I$ for all $i = 1,2,\ldots,d$. Expand \eqref{time_deriv_of_normal} using a product rule and evaluate at $t = t_i$
\begin{equation}
    \left[ N_{u_{\bm r}(\bm x,t)}\left(\frac{\partial u_{\bm r}(\bm x,t)}{\partial t}\right)\right]_{t = t_i} = \bm \Gamma_1(t_i) \cdots \bm \Gamma_{j-1}(t_i) \frac{\partial \bm C_j(t)}{\partial t}\bigg\vert_{t = t_i} \bm \Gamma_j(t_i) \cdots \bm \Gamma_d(t_i).
\end{equation}
From the previous equation we see that the FTT autocorrelation 
matrices of the normal component at time instant $t_i$ 
are the time derivatives of the zero energy modes in the 
current solution. Thus, if the normal component has FTT 
rank $\bm n$ then the solution $u_{\bm r}(\bm x,t)$ 
at time $t_i$ should be represented by an FTT tensor of 
rank $\bm r + \bm n$. Certainly, the solution will be over 
represented at $t_i$ with rank $\bm r+ \bm n$. However, 
after one step of the splitting integrator the additional 
ranks will ensure that the low-rank solution 
$u_{\bm r+\bm n}(\bm x,t) \in \mathcal{M}_{\bm r+\bm n}$ 
retains its accuracy. 

The main steps of the algorithm we propose to 
adaptively increase the tensor rank are summarized 
in Algorithm \ref{alg:adaptive_rank_inc}. The operation ``$*$''
appearing within the conditional statement if/end 
denotes scalar times FTT tensor, and is meant 
to indicate that the multiplication is done by scaling 
the first core of the tensor with the scalar $0$ and leaving 
the remainder of the cores unchanged \cite{OseledetsTT}. 
As we will demonstrate 
in section \ref{sec:numerics}, Algorithm 1 is robust 
and it yields accurate results that do no require ad-hoc 
approximations such the matrix pseudo-inverse approximation 
introduced in \cite{robust_do/bo}.

\vs
\begin{algorithm}[t]
\SetAlgoLined
 \caption{One step integration with adaptive rank increase}
\label{alg:adaptive_rank_inc}
 \KwIn{\\
 \begin{itemize}
 \item[] $u_{\bm r}(\bm x,t_i),u_{\bm r}(\bm x,t_{i-1}), \ldots, u_{\bm r}(\bm x,t_{i-p})$ $\rightarrow$ time snapshots of the PDE solution with rank $\bm r$,
 \item[]  $G(u_{\bm r}(\bm x,t_i))$ $\rightarrow$ velocity vector defined by the right hand side of the PDE \eqref{nonlinear-ibvp} at time $t_i$,
 \item[] $\Delta t$ $\rightarrow$ time step, 
 \item[] $\epsilon_{\mathrm{inc}}$ $\rightarrow$ threshold for the norm of normal component $N_{u_{\bm r}(\bm x,t_i)} (G(u_{\bm r}(\bm x,t_i)))$.
 \end{itemize}
 }
 \KwOut{$u_{\bm r+\bm n}(\bm x,t_{i+1})$ $\rightarrow$  PDE solution with rank $\bm r+\bm n$ at time $t_{i+1}$\vspace{0.2cm}}
{ \bf Initialization:}\\
\begin{itemize}
\item Approximate the constant rank velocity vector via the BDF formula:  \\
 $\widetilde{P}^{(p)}_{u_{\bm r}(\bm x,t_i)} G(u_{\bm r}(\bm x,t_i)) = BD_p(u_{\bm r}(\bm x,t_i),u_{\bm r}(\bm x,t_{i-1}), \ldots, u_{\bm r}(\bm x,t_{i-p}))$ \\
\item  Compute  the normal component: \\
 $N_{u_{\bm r}(\bm x,t_i)} G(u_{\bm r}(\bm x,t_i)) = G(u_{\bm r}(\bm x,t_i)) - \widetilde{P}^{(p)}_{u_{\bm r}(\bm x,t_i)} G(u_{\bm r}(\bm x,t_i))$ 
 \end{itemize}
 { \bf Runtime:}\\
 \begin{itemize}
 \item {\bf if} $\|N_{u_{\bm r}(\bm x,t_i)} G(u_{\bm r}(\bm x,t_i))\| > \epsilon_{\mathrm{inc}}$ {\bf then} \\
 \hspace{0.2cm}Compute the FTT decomposition of normal component: \\
 \hspace{0.2cm}$N_{\text{TT}}(\bm x,t_i) = \text{FTT}(N_{u_{\bm r}(\bm x,t_i)} G(u_{\bm r}(\bm x,t_i)))$ \\
\vspace{0.2cm}\hspace{0.2cm} Initialize to zero  
additional tensor modes in $u_{\bm r}(\bm x,t_i)$, as 
many as the rank of $N_{\text{TT}}$ (say $\bm n$): \\
 \hspace{0.2cm} $u_{\bm r+\bm n}(\bm x,t_i) = u_{\bm r}(\bm x,t_i) + 0*N_{\text{TT}}(\bm x,t_i)$\\
 {\bf end}
 \item Use one step of Lie-Trotter splitting integrator to map $u_{\bm r+\bm n}(\bm x,t_i)$ into $u_{\bm r+\bm n}(\bm x,t_{i+1})$ 
 \end{itemize} 
 
\end{algorithm}
\vs

\subsection{Order of the rank-adaptive tensor scheme}

Let us choose the threshold  $\epsilon_{\mathrm{inc}}$ in 
\eqref{condition_for_rank_inc} to satisfy 
\begin{equation}
    \label{normal_thresh_upper_bound}
    \epsilon_{\mathrm{inc}} \leq K_2 \Delta t, 
\end{equation}
and assume that the condition 
\begin{equation}
   \| N_{u_{\bm r}(\bm x,t)}(G(u_{\bm r}(\bm x,t)))\| \leq \epsilon_{\mathrm{inc}}
\end{equation}
is satisfied for all $ t \in [0,T]$. Then we have the following bound 
for the local truncation error 
\begin{equation}
    \begin{aligned}
        \|E(t_i,t_{i+1})\| &= \left\| \int_{t_i}^{t_{i+1}} N_{u_{\bm r}(\bm x,\tau)}(G(u_{\bm r}(\bm x,\tau))) d\tau \right\| \\
        &\leq \int_{t_i}^{t_{i+1}} \| N_{u_{\bm r}(\bm x,\tau)}(G(u_{\bm r}(\bm x,\tau))) \| d\tau \\
        &\leq \int_{t_i}^{t_{i+1}}K_2 \Delta t d\tau \\
        &= K_2 \Delta t^2.
    \end{aligned}
\end{equation}
In particular, we have that the continuous-time 
rank-adaptive scheme is order one consistent 
in $\Delta t$ if the normal vector threshold 
is set as in \eqref{normal_thresh_upper_bound}. 

When implementing the adaptive scheme we usually 
discretize the time domain $[0,T]$ into a mesh of time instants as 
in \eqref{discrete_time}. Therefore, we do not necessarily have control 
over the normal vector for all $t \in [0,T]$ but rather only at a finite 
number of time instants. However, an analogous 
argument as we have made for order one consistency in the 
continuous time {rank-adaptive} scheme holds for 
the discrete time {rank-adaptive} scheme by considering the 
first-order approximation of the local truncation error given in 
\eqref{left_endpoint_approx_of_normal}. In particular by using 
the equality in \eqref{left_endpoint_approx_of_normal} and 
discrete time thresholding of the normal component
\begin{equation}
\label{discrete_normal_thresholding}
    \| N_{u_{\bm r}(\bm x,t_i)}(G(u_{\bm r}(\bm x,t_i)))\| \leq \epsilon_{\mathrm{inc}}, \qquad \forall i = 0,1,\ldots,N,
\end{equation}
we have that 
\begin{equation}
    \begin{aligned}
        \|E(t_i,t_{i+1})\| &= \left\| \Delta t N_{u_{\bm r}(\bm x,t_i)}(G(u_{\bm r}(\bm x,t_i))) + O(\Delta t^2) \right\| \\
        &\leq \left\| \Delta t N_{u_{\bm r}(\bm x,t_i)}(G(u_{\bm r}(\bm x,t_i)))\right\| + \| O(\Delta t^2) \| \\
        &= K_2 \Delta t^2 + O(\Delta t^2) \\
        &= O(\Delta t^2).
    \end{aligned}
\end{equation}
This proves that the discrete time {rank-adaptive} scheme 
with normal threshold given by \eqref{discrete_normal_thresholding} 
is consistent with order one in $\Delta t$. Higher-order consistency 
results can be obtained with higher-order time integration methods 
and higher-order estimators for the normal vector ${N_{u_{\bm r}}} G(u_{\bm r})$.


\section{Numerical examples}
\label{sec:numerics}
In this section we demonstrate the proposed rank-adaptive 
FTT tensor method on linear and nonlinear PDEs. In all examples 
the rank-adaptive scheme relies on first-order 
Lie-Trotter operator splitting time integration
\eqref{split_diff_eqs}, and the thresholding criterion 
\eqref{condition_for_rank_inc}. For each PDE we rigorously 
assess the accuracy of the proposed rank-adaptive tensor
method by comparing it with benchmark solutions computed 
with well-established numerical methods.

\subsection{Two-dimensional variable coefficient advection equation}
Let us begin with the two-dimensional variable coefficient advection problem
\begin{equation}
 \label{2d_advection}
 \begin{cases}
\displaystyle\frac{\partial u(x_1,x_2,t)}{\partial t} = (\sin(x_1) +
 \cos(x_2))\frac{\partial u(x_1,x_2,t)}{\partial x_1} +
 \cos(x_2)\frac{\partial u(x_1,x_2,t)}{\partial x_2}, \vs\\
u(x_1,x_2,0) = \exp[\sin(x_1 + x_2)],
\end{cases}
\end{equation}
on the flat torus $\Omega=\mathbb{T}^2$. We have shown 
in previous work \cite{Dektor_2020} that the tensor 
solution to the PDE \eqref{2d_advection} 
increases in rank as time increases.
\begin{figure}[t]
\hspace{0.2cm}
\centerline{\footnotesize\hspace{0.1cm}$t = 0.0$ \hspace{4.3cm} $t = 0.5$  \hspace{4.3cm} $t = 1.0$ }

\hspace{0.25cm} 
\centerline{
	\rotatebox{90}{\hspace{1.5cm}\footnotesize Adaptive FTT }
		\includegraphics[width=0.265\textwidth,height=3.9cm]{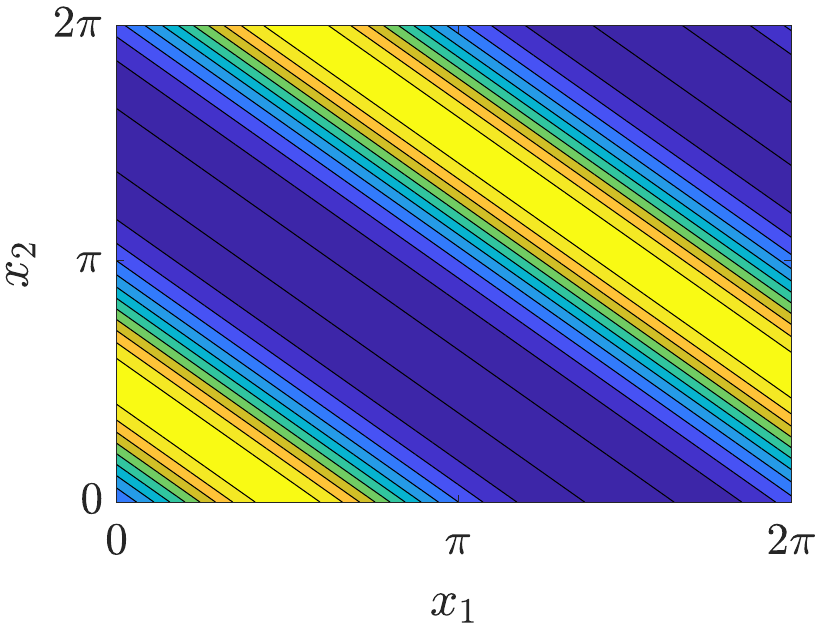}
		\hspace{.8cm}
		\includegraphics[width=0.268\textwidth,height=3.9cm]{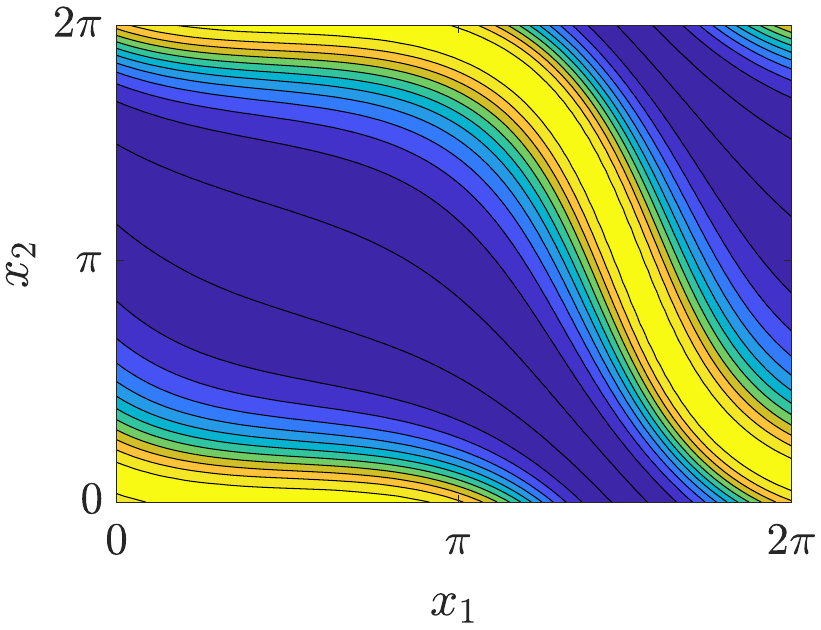}
		\hspace{.75cm}
		\includegraphics[width=0.31\textwidth,height=3.9cm]{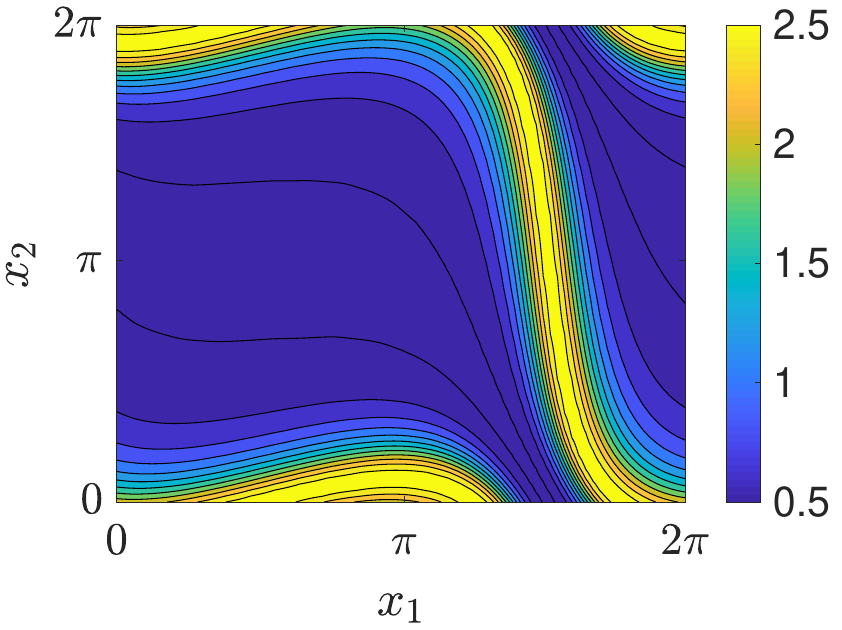}	
		\hspace{.1cm}
}
	
\hspace{0.3cm}	
\centerline{
\rotatebox{90}{\hspace{.7cm}  \footnotesize  Method of characteristics}
       \includegraphics[width=0.265\textwidth,height=3.9cm]{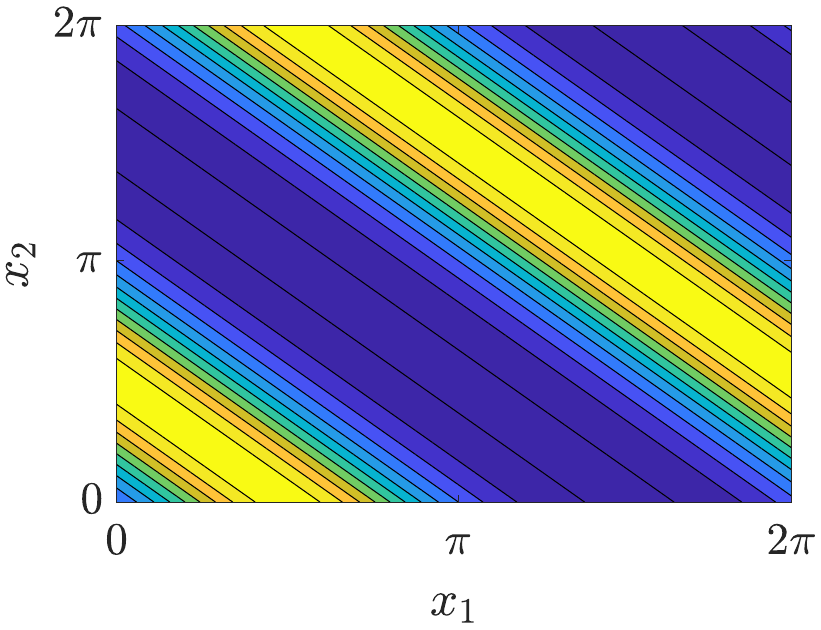}  \hspace{.8cm}
	    \includegraphics[width=0.268\textwidth,height=3.9cm]{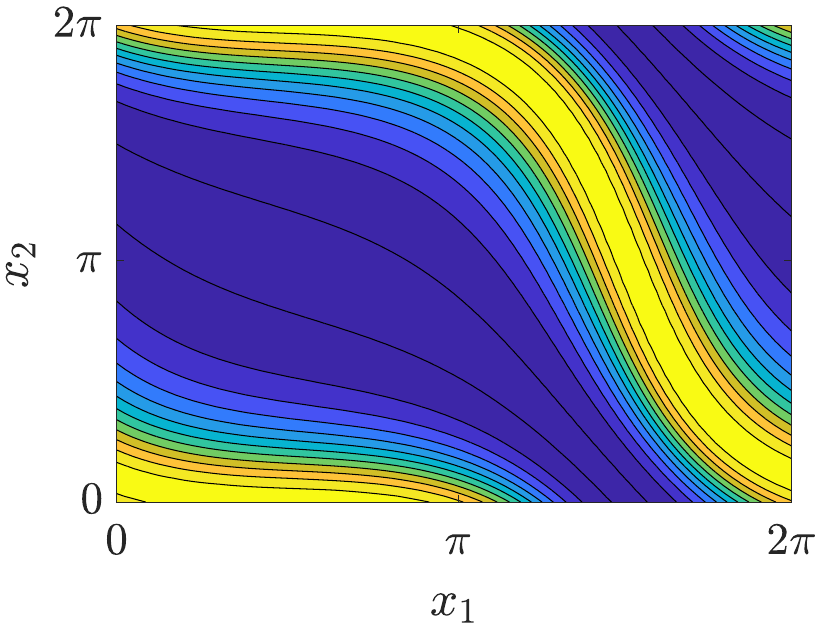}
	    \hspace{.75cm}
	    \includegraphics[width=0.31\textwidth,height=3.9cm]{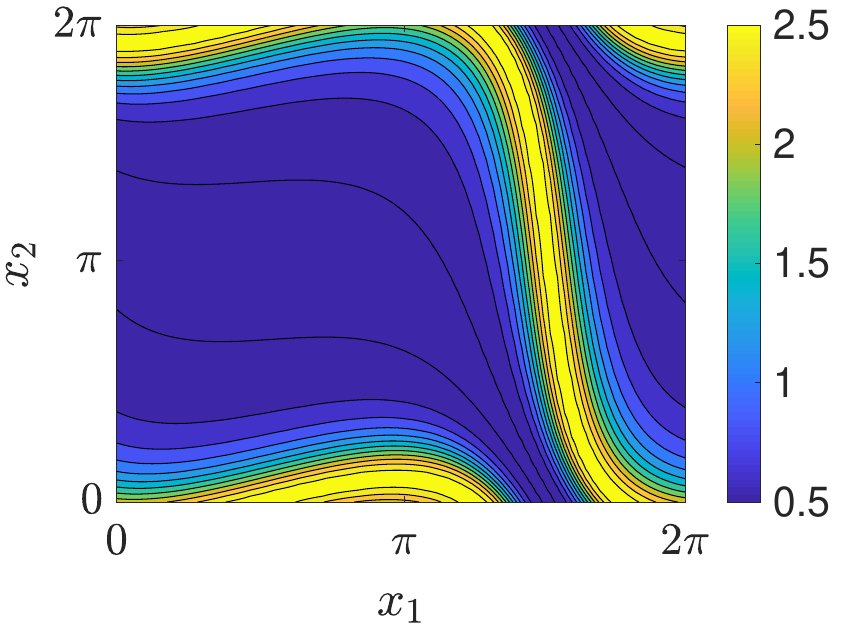}
	    \hspace{.1cm}
}

\hspace{0.22cm}	
\centerline{
	\rotatebox{90}{\hspace{1.2cm}  \footnotesize Pointwise error}
       \includegraphics[width=0.323\textwidth,height=4.3cm]{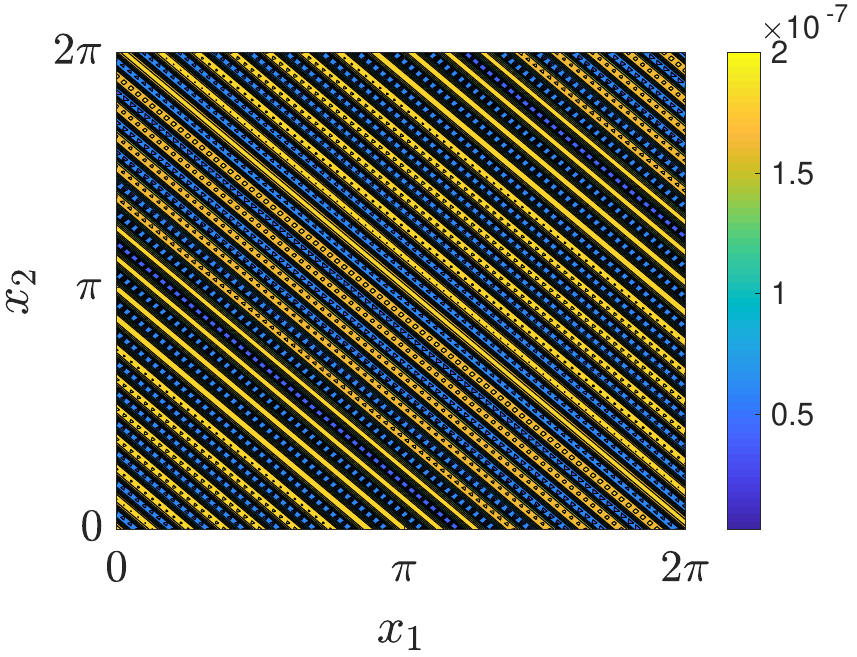} \hspace{-.1cm}
	    \includegraphics[width=0.325\textwidth,height=4.3cm]{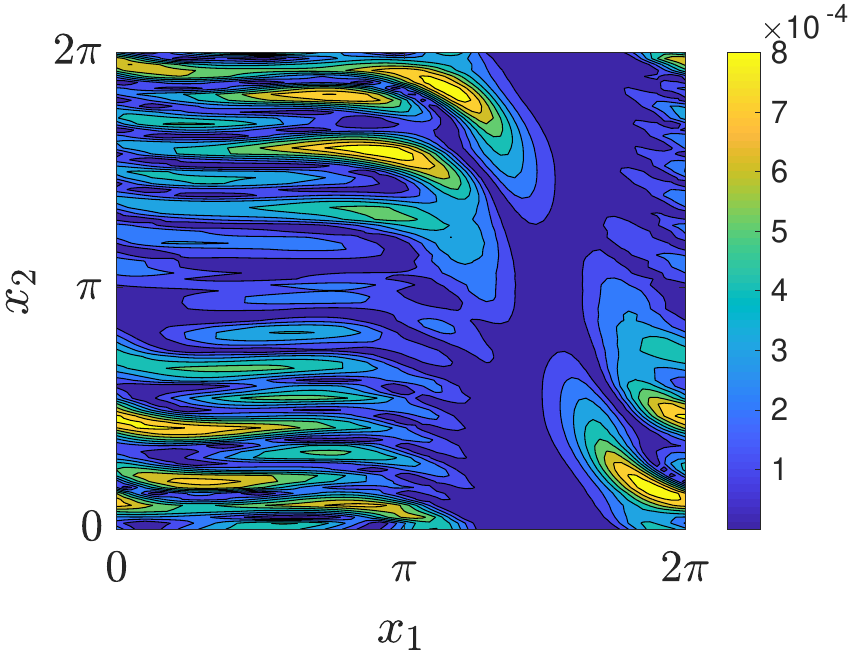}\hspace{-0.05cm}
	    \includegraphics[width=0.315\textwidth,height=4.3cm]{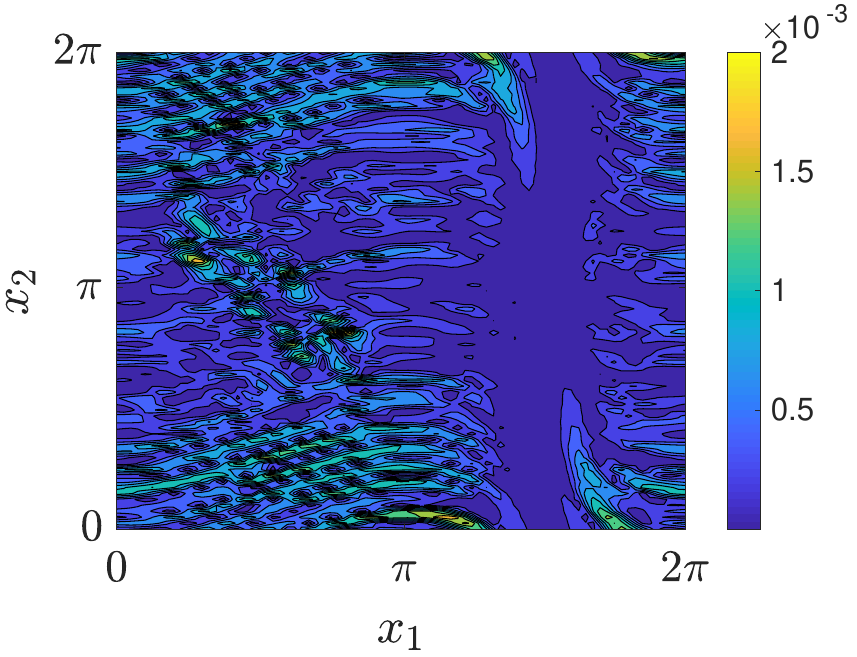}
}
\caption{Variable coefficient advection equation
\eqref{2d_advection}. Time snapshots of the rank-adaptive FTT solution 
${u_{\bm r}}(x_1,x_2,t)$ obtained with threshold 
$\epsilon_{\mathrm{inc}} = 10^{-2}$ (top), the 
semi-analytical solution $u_{\text{ref}}(x_1,x_2,t)$ (middle), and 
the pointwise error between the two solutions (bottom).}
\label{fig:2D_advection_solution_time_evolution}
\end{figure}
As is well known, the PDE \eqref{2d_advection} can be 
reduced to the trivial ODE $du/dt=0$ along the flow generated 
by the dynamical system (see, e.g.,  \cite{Rhee}) 
\begin{equation}
\label{char_system}
\left\{
\begin{aligned}
\displaystyle\frac{d x_1}{d t} &= \sin(x_1) + \cos(x_2), \vs\\
\displaystyle\frac{d x_2}{d t} &= \cos(x_2).
\end{aligned}\right.
\end{equation}
With the flow $\{x_1(t,x_{01},x_{02}),x_1(t,x_{01},x_{02})\}$ 
available, we can write the analytical solution to \eqref{2d_advection} as
\begin{equation}
u_{\mathrm{ref}}(x_1,x_2,t)=\exp\left[\sin(x_{01}(x_1,x_2,t) + x_{02}(x_1,x_2,t))\right],
\label{2Dsol}
\end{equation}
where $\{x_{01}(x_1,x_2,t), x_{02}(x_1,x_2,t)\}$ denotes 
the inverse flow generated by \eqref{char_system}. 
We obtain a semi-analytical solution to the 
PDE \eqref{2d_advection} by solving the characteristic system 
\eqref{char_system} numerically for different initial conditions 
and then evaluating \eqref{2Dsol}. A few time snapshots of the 
semi-analytical solution \eqref{2Dsol} are plotted 
in Figure \ref{fig:2D_advection_solution_time_evolution} (middle row).

\begin{figure}[t]
\centerline{\footnotesize\hspace{0.40cm} (a) \hspace{5cm} (b) \hspace{5cm} (c) }
	\centering
		\includegraphics[width=0.33\textwidth]{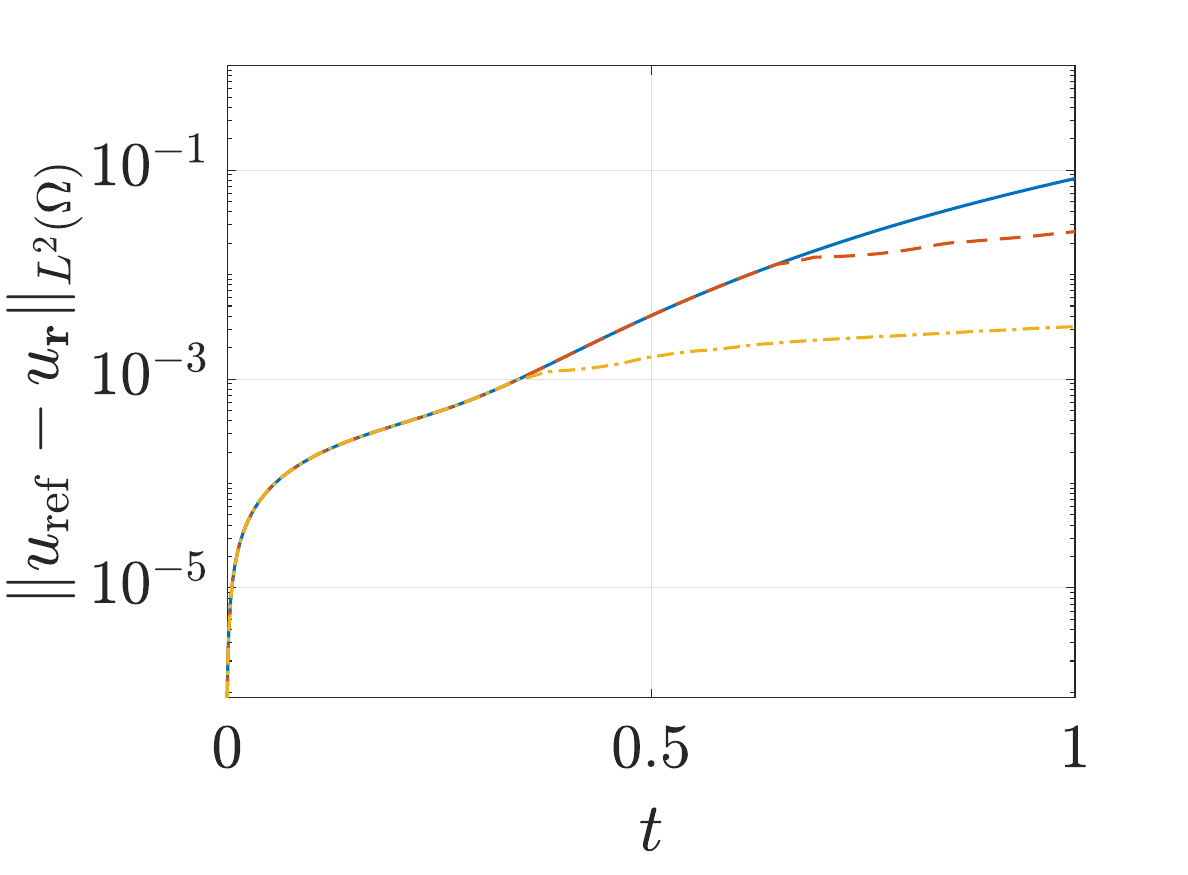}
		\includegraphics[width=0.33\textwidth]{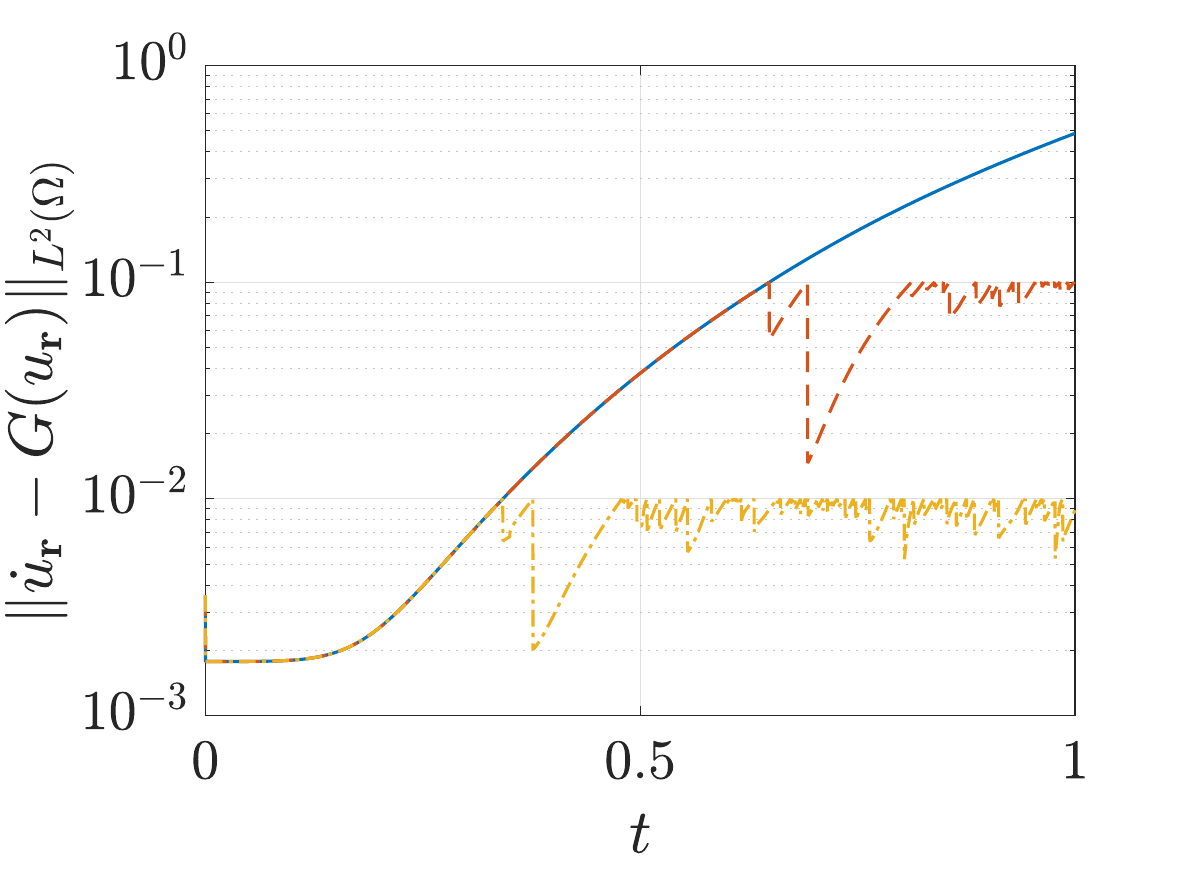}
		\includegraphics[width=0.33\textwidth]{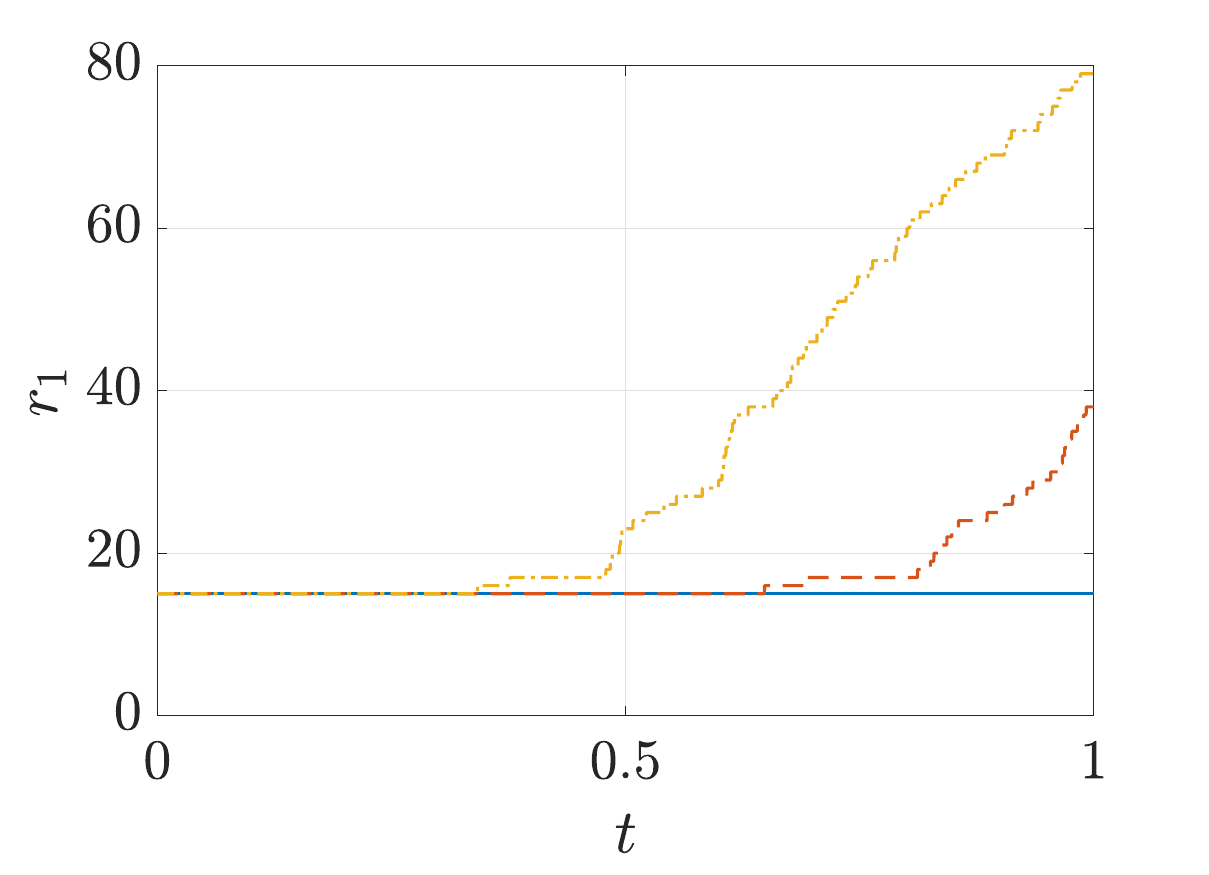}	\\
		\hspace{0.8cm}\includegraphics[scale=0.5]{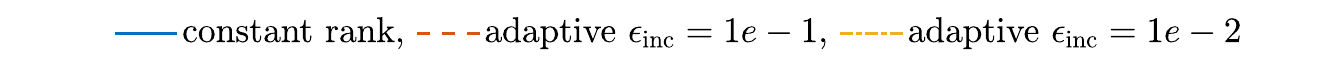}
\caption{
(a) Global $L^2(\Omega)$ error of the FTT solution 
${u_{\bm r}}$ relative to the benchmark solution 
$u_{\mathrm{ref}}$; 
(b) Norm of the two-point BDF approximation to the normal 
component $N_{{u_{\bm r}}}G({u_{\bm r}}(\bm x,t))$ 
(note the effect of thresholding); 
(c) Tensor rank versus time of the constant-rank FTT solution 
and adaptive rank solutions with 
$\epsilon_{\mathrm{inc}} = 10^{-1}$ and 
$\epsilon_{\mathrm{inc}} = 10^{-2}$.}
\label{fig:2d_advection_errors}
\end{figure}
We also solve the PDE \eqref{2d_advection} using the proposed 
rank-adaptive tensor method with first-order Lie-Trotter 
operator splitting and thresholding criterion 
\eqref{condition_for_rank_inc} with $\epsilon_{\text{inc}}=10^{-2}$. 
The initial condition is approximated by an FTT 
tensor ${u_{\bm r}}(x_1,x_2,0)$
with multivariate rank $\bm r = \begin{bmatrix} 1 & 15 & 1
\end{bmatrix}$ 
\begin{equation}
\label{2D_advection_TT_IC}
    {u_{\bm r}}(x_1,x_2,0) = \bm \Psi_1(x_1) 
    \sqrt{\bm \Lambda} \bm \Psi_2(x_2),
\end{equation}
where 
\begin{equation}
    \begin{aligned}
    \bm \Psi_1(x_1) = \begin{bmatrix} \psi_1(1;x_1;1) & \cdots & \psi_1(1;x_1;15) \end{bmatrix} , \quad 
    \sqrt{\bm \Lambda} = \begin{bmatrix} \sigma_1 &  &  \\
                             & \ddots &  \\
                             &  & \sigma_{15} 
                            \end{bmatrix}, \quad 
    \bm \Psi_2(x_2) = \begin{bmatrix} \psi_2(1;x_1;1) \\  \vdots \\ \psi_2(15;x_2;1) \end{bmatrix}.
    \end{aligned}
\end{equation}
Each tensor mode $\psi_i$ is discretized 
on a grid of $81$ evenly-spaced points in the interval 
$\Omega_i = [0,2\pi]$. One-dimensional Fourier pseudo-spectral 
quadrature rules and differentiation matrices 
\cite{spectral_methods_book} are used to 
compute inner products and derivatives when needed. 
We run three simulations with the initial 
tensor decomposition \eqref{2D_advection_TT_IC} 
and time step $\Delta t = 10^{-4}$. In the first simulation we do 
not use any rank adaptation, in the second simulation 
we set the normal vector threshold to 
$\epsilon_{\mathrm{inc}} = 10^{-1}$ and in the third 
simulation we set $\epsilon_{\mathrm{inc}} = 10^{-2}$.
At each time step the component of 
$G({u_{\bm r}}(\bm x,t_i))$
normal to the tensor manifold 
is approximated with the two-point
BDF formula (section \ref{sec:increase_rank}). 
In Figure \ref{fig:2d_advection_spectrum} we plot 
a few time snapshots of the singular values of 
the {rank-adaptive} FTT solution with 
$\epsilon_{\text{inc}} = 10^{-2}$.
\begin{figure}[t]
\centerline{\footnotesize\hspace{0.85cm} t = 0.0 \hspace{4.75cm} t = 0.5 \hspace{4.5cm} t = 1.0 }
	\centering
		\includegraphics[width=0.33\textwidth]{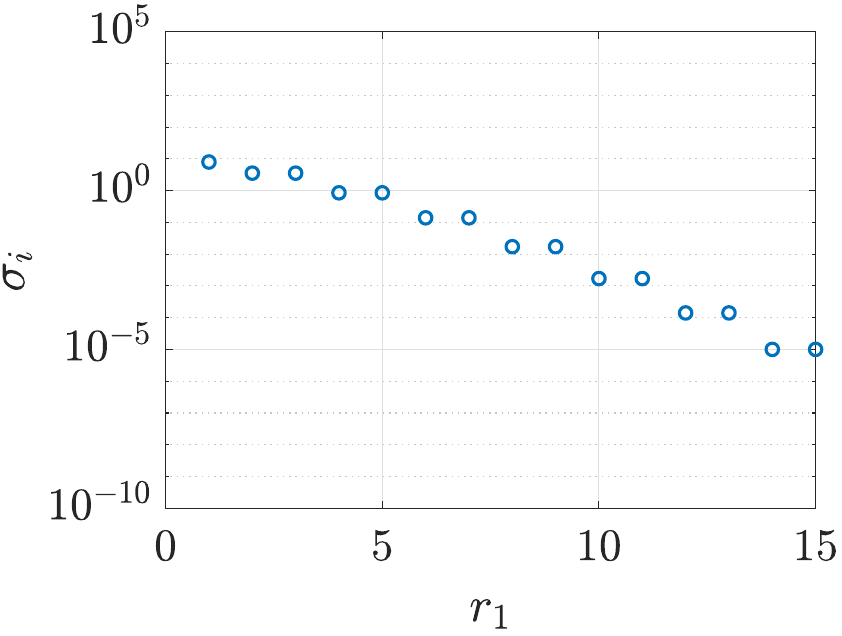}
		\includegraphics[width=0.33\textwidth]{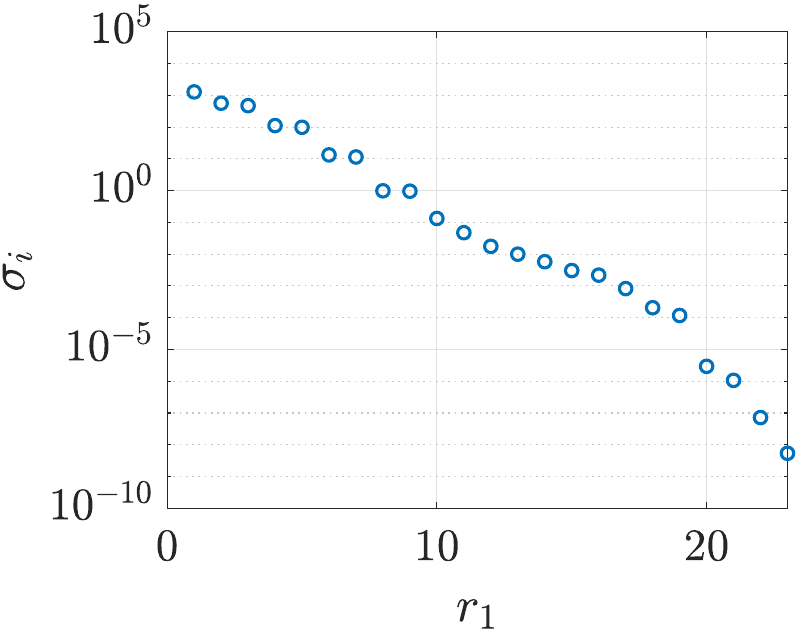}
		\includegraphics[width=0.33\textwidth]{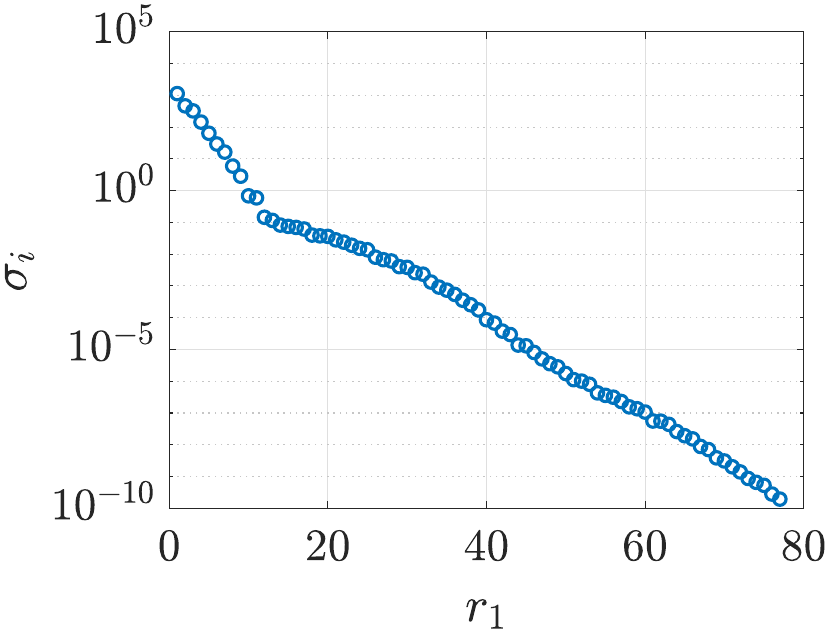}
\caption{
Time snapshots of the singular values of the rank-adaptive FTT solution 
with threshold $\epsilon_{\mathrm{inc}} = 10^{-2}$.}
\label{fig:2d_advection_spectrum}
\end{figure}
Figures \ref{fig:2d_advection_errors}(a)-(c) summarize 
the performance and accuracy of the proposed 
rank-adaptive FTT solver. 
In particular, in  Figure \ref{fig:2d_advection_errors}(a)
we plot the time-dependent $L^2(\Omega)$ error between 
the rank-adaptive FTT solution and the reference 
solution we obtained with method of characteristics. It is 
seen that decreasing the threshold $\epsilon_{\text{inc}}$ 
on the norm of the component of $G({u_{\bm r}})$ normal 
to the FTT tensor manifold  (Figure \ref{fig:2d_advection_errors}(b)) 
yields addition of more tensor mores to the FTT solution 
(Figure \ref{fig:2d_advection_errors}(c)). This, in turn, results 
in better accuracy as demonstrated in 
Figure \ref{fig:2d_advection_errors}(a).

\subsection{Two-dimensional Kuramoto-Sivashinsky equation}
In this section we demonstrate the rank-adaptive FTT 
integrator on the two-dimensional Kuramoto-Sivashinsky 
equation \cite{KSE}
\begin{equation}
    \label{transformed_KS}
    \begin{cases}
    \displaystyle\frac{\partial}{\partial t}u(x_1,x_2,t) + \frac{1}{2} |\nabla_{\nu} u(x_1,x_2,t)|^2 + \Delta_{\nu} u(x_1,x_2,t) + \nu_1 \Delta_{\nu}^2 u(x_1,x_2,t) = 0,\vs\\
    u(x_1,x_2,0) = \sin(x_1 + x_2) + \sin(x_1) + \sin(x_2),
    \end{cases}
\end{equation}
where 
\begin{equation}
\label{transformed_operators}
\nabla_{\nu} = \left( \frac{\partial}{\partial x_1}, 
{\frac{\nu_2}{\nu_1}}
\frac{\partial}{\partial x_2} \right), \qquad \Delta_{\nu} 
= \frac{\partial^2 }{\partial x_1^2} + 
{\frac{\nu_2}{\nu_1}} 
\frac{\partial^2}{\partial x_2^2}.
\end{equation}
Here, $\nu_1$, $\nu_2$ are bifurcation parameters. 
For our demonstration we set $\nu_1 = 0.25$, $\nu_2 = 0.04$ and solve 
\eqref{transformed_KS} on the two-dimensional 
flat torus $\mathbb{T}^2$. 
{
The initial condition can be written as rank $\bm r = \begin{bmatrix}
1 & 2 & 1
\end{bmatrix}$ FTT tensor 
\begin{equation}
u_0(x_1,x_2) = \psi_1(1;x_1;1) \psi_2(1;x_2;1) \sqrt{\lambda(1)} + \psi_1(1;x_1;2) \psi_2(2;x_2;1) \sqrt{\lambda(2)},
\end{equation}
where 
\begin{equation}
\begin{aligned}
\psi_1(1;x_1;1) &= \frac{\sin(x_1)}{\sqrt{\pi}} ,\qquad\quad\hspace*{.25cm} \psi_1(1;x_1;2) = \frac{\cos(x_1) + 1}{ \sqrt{3\pi}} , \\
\psi_2(1;x_2;1) &= \frac{\cos(x_2)+1}{\sqrt{3\pi}} ,\qquad \psi_2(2;x_2;1) = \frac{\sin(x_2)}{\sqrt{\pi}} ,
\end{aligned}
\end{equation}
and
\begin{equation}
\sqrt{\lambda(1)} = \sqrt{\lambda(2)} = \sqrt{3}\pi.
\end{equation}
}
\begin{figure}[t]
\hspace{0.2cm}
\centerline{\footnotesize\hspace{-0.2cm}$t = 0.0$ \hspace{4.8cm} $t = 2.5$  \hspace{4.9cm} $t = 5.0$ }

\hspace{0.1cm} 
\centerline{
	\rotatebox{90}{\hspace{1.5cm}\footnotesize Adaptive FTT }
		\includegraphics[width=0.32\textwidth]{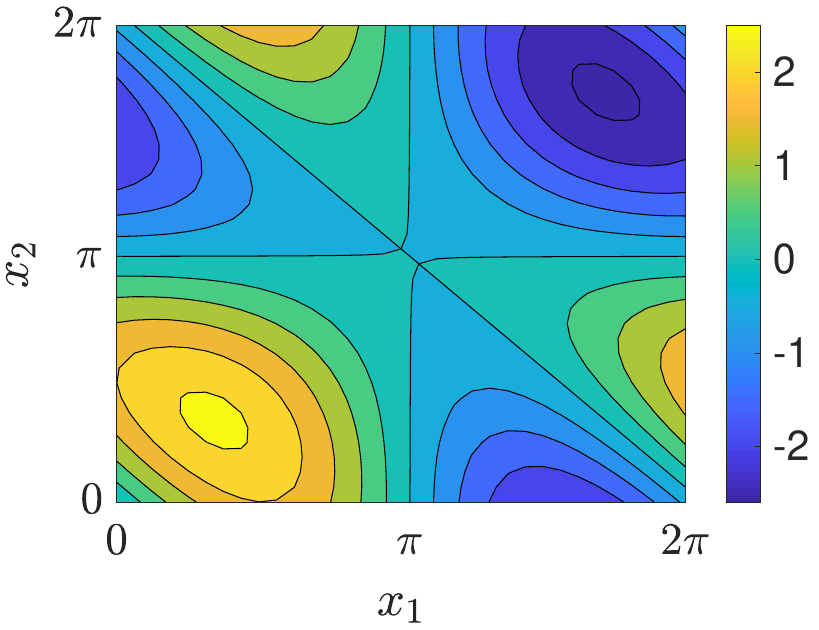}
		\hspace{.5cm}
		\includegraphics[width=0.32\textwidth]{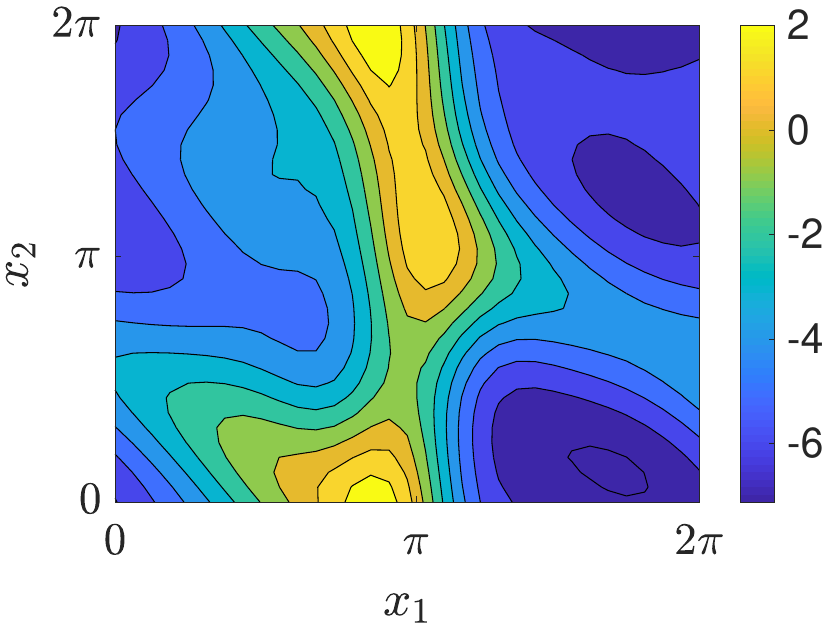}
		\hspace{.5cm}
		\includegraphics[width=0.32\textwidth]{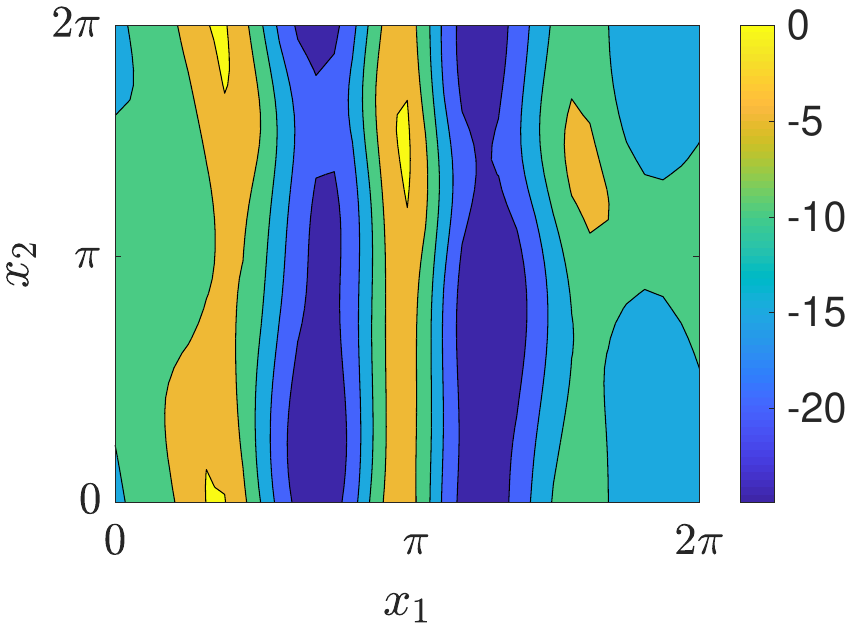}	
		\hspace{.2cm}
}
	
\hspace{0.1cm}	
\centerline{
\rotatebox{90}{\hspace{1.0cm}  \footnotesize  Full tensor product}
       \includegraphics[width=0.32\textwidth]{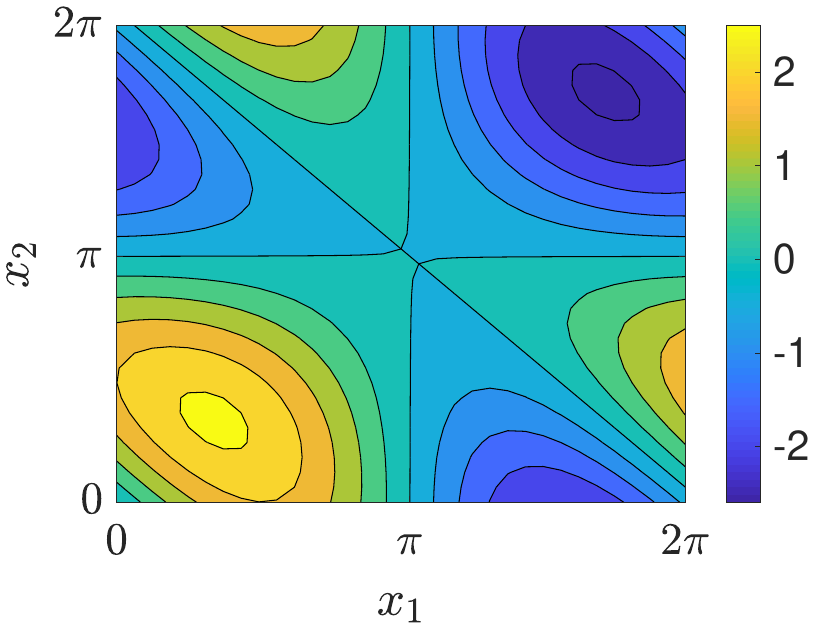}
       \hspace{.5cm}
	    \includegraphics[width=0.32\textwidth]{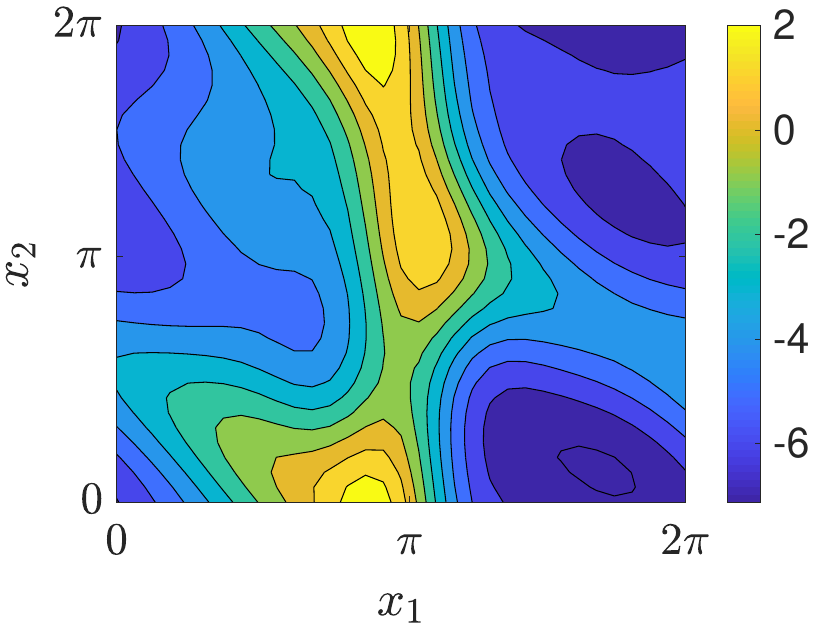}
	    \hspace{.5cm}
	    \includegraphics[width=0.32\textwidth]{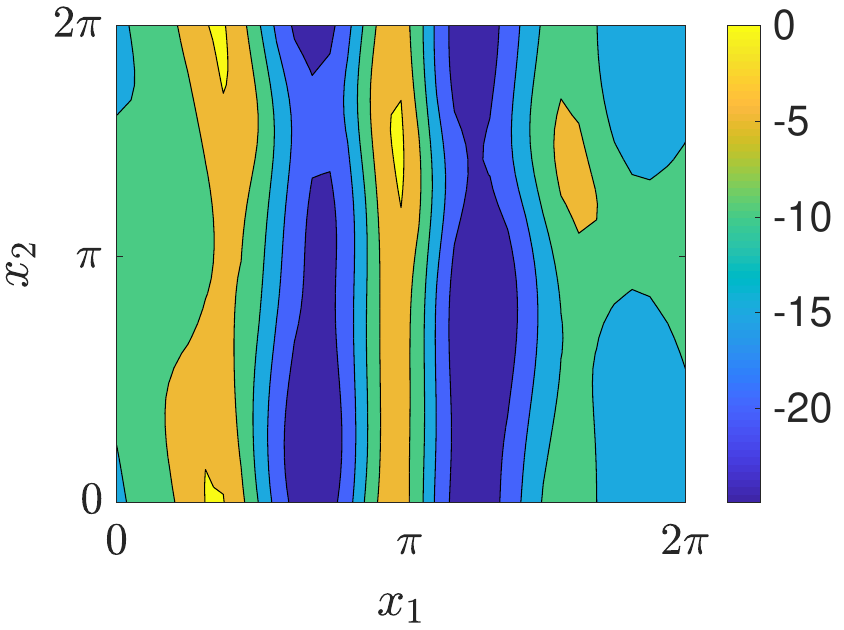}
	    \hspace{.2cm}
}

\hspace{0.cm}	
\centerline{
	\rotatebox{90}{\hspace{1.2cm}  \footnotesize Pointwise error}
  \hspace{0.0cm}     \includegraphics[width=0.34\textwidth,height=4.4cm]{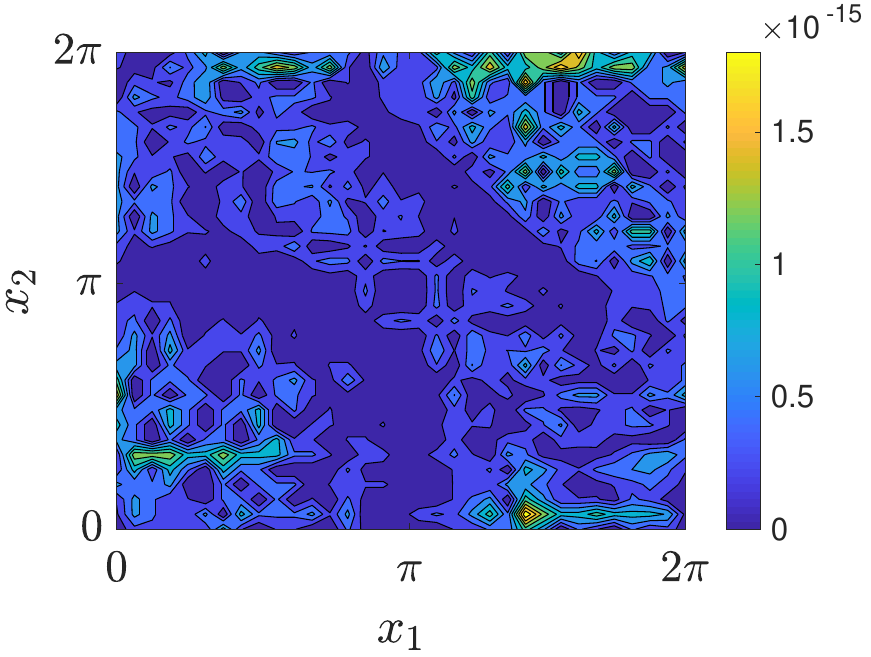}\hspace{0.2cm}
	    \includegraphics[width=0.34\textwidth,height=4.4cm]{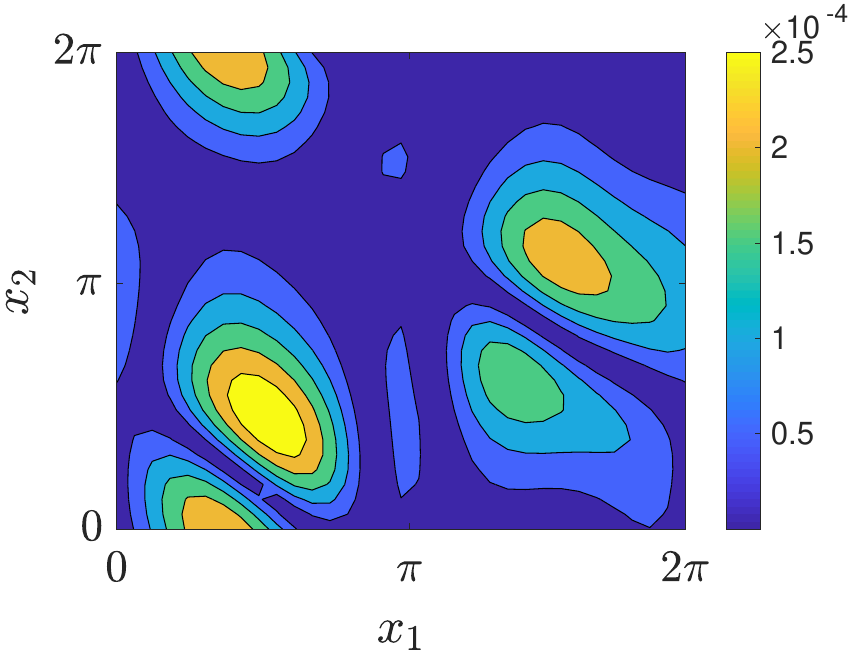}\hspace{0.25cm}
	    \includegraphics[width=0.33\textwidth,height=4.4cm]{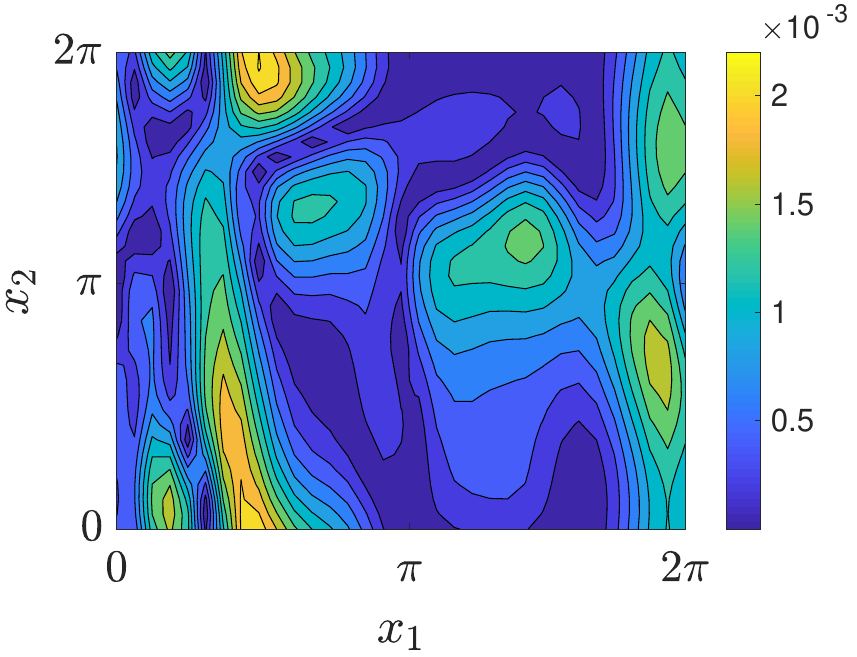}
}
\caption{Kuramoto-Sivashinsky equation \eqref{transformed_KS}. 
Time snapshots of the rank-adaptive FTT solution ${u_{\bm r}}(x_1,x_2,t)$ obtained with threshold $\epsilon_{\mathrm{inc}} = 10^{-2}$ 
(top), the Fourier pseudo-spectral solution $u_{\mathrm{ref}}(x_1,x_2,t)$ (middle), and their pointwise error between the two solutions (bottom).}
\label{fig:2D_KSE_time_evolution}
\end{figure}
We compute a benchmark solution by using 
a Fourier pseudo-spectral method \cite{spectral_methods_book} 
{
with $33$ evenly-spaced grid points 
per spatial dimension (1089 total number of points)}. 
Derivatives and integrals are approximated with well-known 
pseudo-spectral differentiation matrices and Gauss quadrature 
rules. The resulting ODE system is integrated 
forward in time using an explicit fourth-order 
Runge-Kutta method with time step $\Delta t = 10^{-5}$. 

As before, we performed multiple simulations using the proposed 
rank-adaptive FTT algorithm with different thresholds for the
component of $G({u_{\bm r}})$ normal to the 
tensor manifold. Specifically, we ran one simulation with 
no mode addition and three simulations with adaptive 
mode addition based on Algorithm 1, and thresholds 
set to $\epsilon_{\mathrm{inc}} = 10$, 
$\epsilon_{\mathrm{inc}} = 10^{-1}$, and 
$\epsilon_{\mathrm{inc}} = 10^{-2}$.
We used the two-point BDF formula 
\eqref{backwards_fd_1st_order} to approximate 
the component of the solution normal to the 
tensor manifold at each time step and the Lie-Trotter 
operator splitting scheme \eqref{split_diff_eqs} with 
time step $\Delta t = 10^{-5}$ to integrate in time 
the rank-adaptive FTT solution.
In Figure \ref{fig:2D_KSE_time_evolution} we compare 
the time snapshots of the rank-adaptive FTT solution 
with $\epsilon_{\mathrm{inc}} = 10^{-2}$ with 
the benchmark solution obtained by the Fourier 
pseudo-spectral method. 
As before, Figures \ref{fig:2d_KSE_errors}(a)-(c) demonstrate 
that the rank-adaptive FTT algorithm is effective in controlling 
the $L^2(\Omega)$ error of the FTT solution. Interestingly, 
the solution to the PDE \eqref{transformed_KS} has the property 
that any tensor approximation with sufficient rank yields a 
normal component that does not grow in time. In fact, as seen 
in Figure   \ref{fig:2d_KSE_errors}(b) the tensor rank becomes 
constant for each threshold $\epsilon_{\text{inc}}$ after 
a transient of approximately 0.5 dimensionless time units. 

{
In Figure \ref{fig:2d_KSE_errors} we observe that the error associated with the constant rank 2 FTT solution increases 
significantly during temporal integration. This suggests that projecting the nonlinear Kuramoto-Sivashinsky 
equation \eqref{transformed_KS} onto a rank 2 FTT manifold yields a reduced-order PDE 
which does not accurately capture the dynamics of the full system. 
A similar phenomenon occurs in other areas of reduced-order modeling, 
e.g., when projecting nonlinear PDEs onto proper orthogonal decomposition (POD) bases \cite{Sirisup}.
}

\begin{figure}[t]
\centerline{\footnotesize\hspace{0.45cm} (a) \hspace{5cm} (b) \hspace{4.8cm} (c) }
	\centering
		\includegraphics[width=0.33\textwidth]{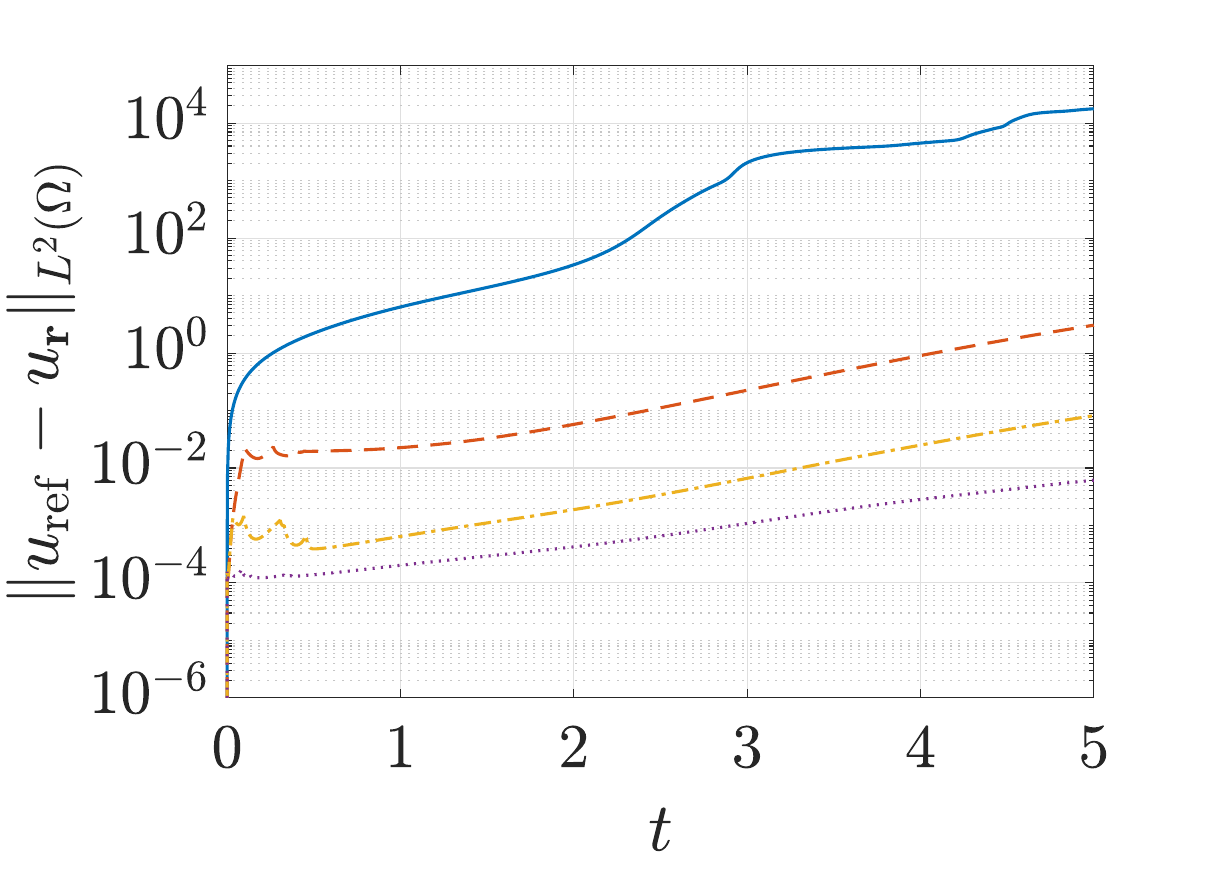}
		\includegraphics[width=0.33\textwidth]{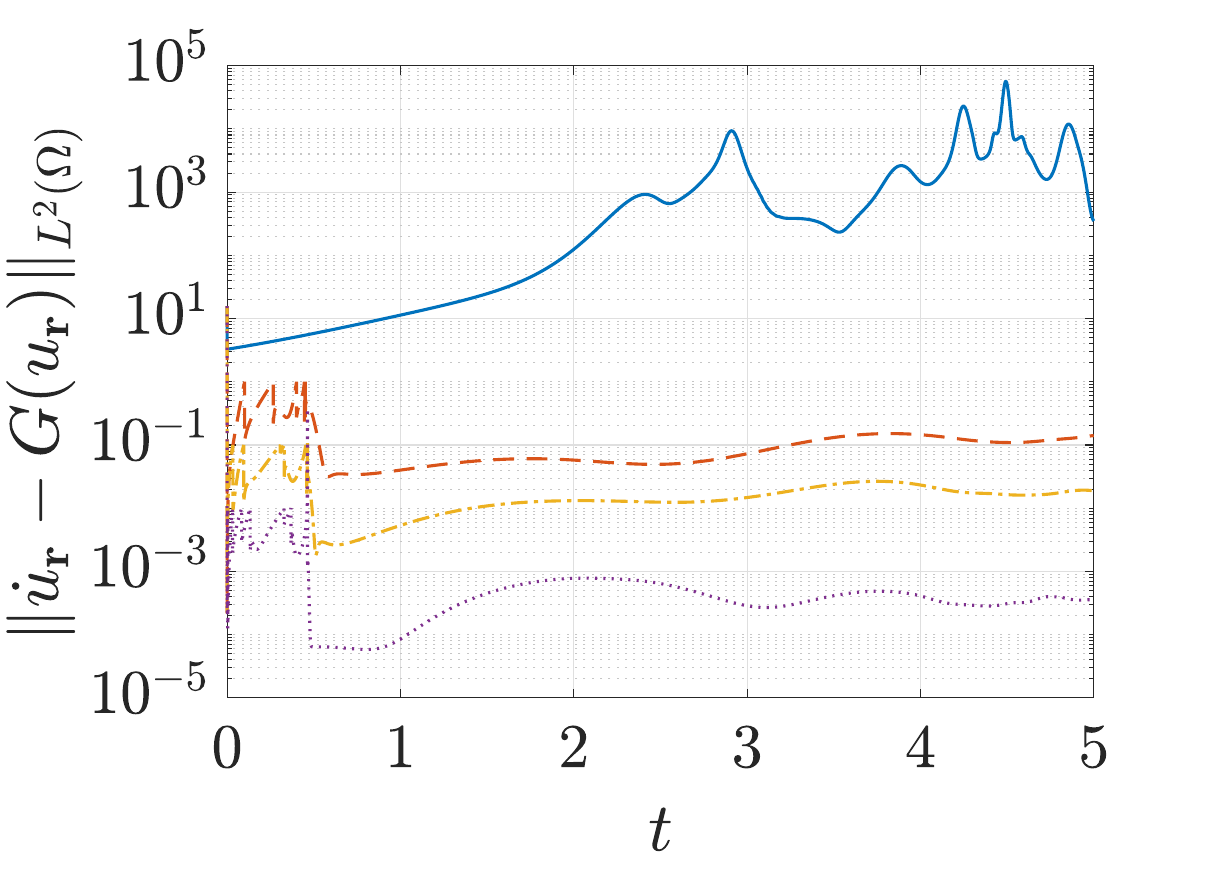}
		\includegraphics[width=0.32\textwidth]{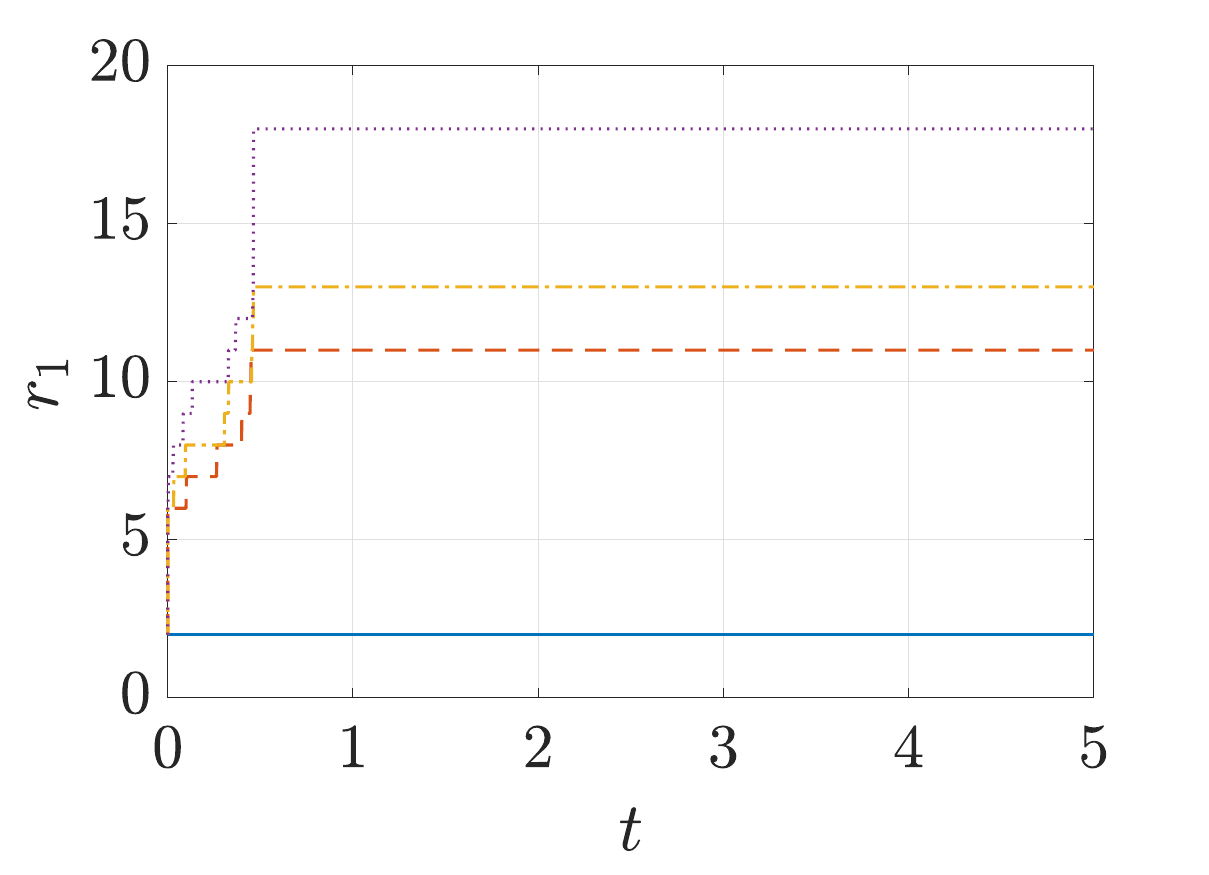}	\\
		\hspace{0.8cm}\includegraphics[scale=0.5]{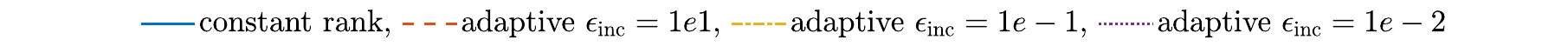}
\caption{
(a) Global $L^2(\Omega)$ error between the FTT solution 
${u_{\bm r}}$ to equation \eqref{transformed_KS} 
and the benchmark solution $u_{\mathrm{ref}}$. 
(b) Norm of the approximation to $N_{{u_{\bm r}}}(G({u_{\bm r}})) =  {\dot{u}_{\bm r}} - G({u_{\bm r}})$ where the tangent space projection is 
computed with a two-point BDF formula at each time. 
(c) Rank versus time of the constant rank FTT solution and rank-adaptive FTT solutions with $\epsilon_{\mathrm{inc}} = 10,10^{-1},10^{-2}$.}
\label{fig:2d_KSE_errors}
\end{figure}

\subsection{Four-dimensional Fokker-Planck equation}
Finally, we demonstrate the 
proposed  rank-adaptive FTT integrator 
on a four-dimensional Fokker--Planck equation with 
non-constant drift and diffusion coefficients. 
As is well known \cite{Risken},
the Fokker--Planck equation describes the evolution 
of the probability density function (PDF) of the state vector
solving the It\^o stochastic differential equation (SDE)
\begin{equation}
\label{Ito_SDE}
d \bm X_t = \bm \mu(\bm X_t,t)dt + \bm \sigma(\bm X_t,t)d \bm W_t.
\end{equation}
Here, $\bm X_t$ is the $d$-dimensional state vector, $\bm \mu(\bm X_t,t)$ is the $d$-dimensional drift, $\bm \sigma(\bm X_t,t)$ is an $d \times m$ 
matrix and $\bm W_t$ is an $m$-dimensional standard 
Wiener process. The Fokker--Planck equation that 
corresponds to \eqref{Ito_SDE} has the form 
\begin{equation}
\label{Fokker_Planck}
\begin{cases}
\displaystyle\frac{\partial p(\bm x,t)}{\partial t} = \mathcal{L}(\bm x,t) p(\bm x,t),\vs  \\
p(\bm x,0) = p_0(\bm x),
\end{cases}
\end{equation}
where $p_0(\bm x)$ is the PDF of the initial state $\bm X_0$,
$\mathcal{L}$ is a second-order linear differential operator 
defined as
\begin{equation}
\mathcal{L}(\bm x,t)p(\bm x,t) = -\sum_{k=1}^d 
\frac{\partial }{\partial x_k}\left( \mu_k(x,t)p(\bm x,t)\right)+ 
\sum_{k,j=1}^d \frac{\partial^2}{\partial x_k\partial x_j}
\left(  D_{ij}(\bm x,t)p(\bm x,t)\right),
\label{L}
\end{equation}
and $\bm D(\bm x,t)=\bm \sigma(\bm x,t) \bm \sigma(\bm x,t)^{\top}/2$ is the diffusion 
tensor. 
For our numerical demonstration we set 
\begin{equation}
\label{drift_diffusion}
\bm \mu(\bm x) = \alpha \begin{bmatrix}
\sin(x_1) \\
\sin(x_3) \\
\sin(x_4) \\
\sin(x_1)
\end{bmatrix}, \qquad
\bm \sigma(\bm x) = \sqrt{2 \beta} \begin{bmatrix}
g(x_2) & 0 & 0 & 0 \\
0 & g(x_3) & 0 & 0 \\
0 & 0 & g(x_4) & 0 \\
0 & 0 & 0 & g(x_1) \\ 
\end{bmatrix},
\end{equation}
where $g(x)=\sqrt{1 + k \sin(x)}$. 
With the drift and diffusion matrices 
chosen in \eqref{drift_diffusion} the 
operator \eqref{L} takes the form  
\begin{equation}
\label{Fokker_Planck_op}
\begin{aligned}
\mathcal{L} =& -\alpha \left(\cos(x_1) + \sin(x_1) \frac{\partial }{\partial x_1} + \sin(x_3) \frac{\partial }{\partial x_2} + \sin(x_4) \frac{\partial }{\partial x_3} + \sin(x_1) \frac{\partial }{\partial x_4} \right) \\
&+ \beta \left( (1 + k \sin(x_2) ) \frac{\partial^2 }{\partial x_1^2} 
+ (1 + k \sin(x_3) ) \frac{\partial^2 }{\partial x_2^2} + 
(1 + k \sin(x_4) ) \frac{\partial^2 }{\partial x_3^2} + 
(1 + k \sin(x_1) ) \frac{\partial^2 }{\partial x_4^2} \right).
\end{aligned}
\end{equation}
Clearly $\mathcal{L}$ is a linear, time-independent separable operator of 
rank $9$, since it can be written as
\begin{equation}
\label{separable_operator}
\mathcal{L} = \sum_{i=1}^9 L_i^{(1)} \otimes L_i^{(2)} 
\otimes L_i^{(3)} \otimes L_i^{(4)} ,
\end{equation}
where each $L_i^{(j)}$ operates on $x_j$ 
only. Specifically, we have  
\begin{equation}
\begin{array}{llll}
L_1^{(1)} = -\alpha \cos(x_1),  &\displaystyle L_2^{(1)} = -\alpha \sin(x_1) \frac{\partial}{\partial x_1},  & \displaystyle L_3^{(2)} = -\alpha \frac{\partial}{\partial x_2},  
&L_3^{(3)} = \sin(x_3), \vspace{0.1cm}\\
L_4^{(3)} =\displaystyle  -\alpha \frac{\partial}{\partial x_3},  
&L_4^{(4)} = \sin(x_4),  &L_5^{(1)} = -\alpha \sin(x_1),  &L_5^{(4)} = \displaystyle \frac{\partial}{\partial x_4}, 
\vspace{0.1cm}\\ 
L_6^{(1)} = \displaystyle \beta \frac{\partial^2}{\partial x_1^2}, & L_6^{(2)} = 1+k\sin(x_2),  &L_7^{(2)} = \displaystyle \beta \frac{\partial^2}{\partial x_2^2}, & L_7^{(3)} = 1+k\sin(x_3), 
\vspace{0.1cm}\\
L_8^{(3)} = \displaystyle \beta \frac{\partial^2}{\partial x_3^2},  
&L_8^{(2)} = 1+k\sin(x_4),  &L_9^{(4)} = \displaystyle \beta \frac{\partial^2}{\partial x_4^2},  &L_9^{(1)} = 1+k\sin(x_1),
\end{array}
\end{equation}
and all other unspecified $L_i^{(j)}$ are identity 
operators. We set the parameters 
in \eqref{drift_diffusion} as 
$\alpha = 0.1$, $\beta = 2.0$, $k = 1.0$ and 
solve \eqref{Fokker_Planck} on the four-dimensional flat
torus $\mathbb{T}^4$.
The initial PDF is set as 
{
\begin{equation}
\label{PDF_IC}
p_0(\bm x) =\frac{\sin(x_1) \sin(x_2) \sin(x_3) \sin(x_4) + 1 }{16\pi^4}.
\end{equation}
Note that \eqref{PDF_IC} is a four-dimensional FTT tensor with multilinear rank 
$\bm r = \begin{bmatrix} 1  & 2 & 2 & 2 & 1 \end{bmatrix}$.
Upon normalizing the modes appropriately we obtain the left orthogonalized 
initial condition required to begin integration
\begin{equation}
\begin{aligned}
p_0(\bm x) &= \psi_1(1;x_1;1) \psi_2(1;x_2;1) \psi_3(1;x_3;1) \psi_4(1;x_4;1) \sqrt{\lambda(1)}  \\
&+\psi_1(1;x_1;2) \psi_2(2;x_2;2) \psi_3(2;x_3;2) \psi_4(2;x_4;1) \sqrt{\lambda(2)},
\end{aligned}
\end{equation}
where 
\begin{equation}
\begin{aligned}
\psi_i(1;x_i;1) = \frac{\sin(x_i)}{\sqrt{\pi}}, \qquad \sqrt{\lambda(1)} = \frac{1}{16 \pi^2}.
\end{aligned}
\end{equation}
All other tensor modes are equal to $1/\sqrt{2 \pi}$, and $\sqrt{\lambda(2)} = 1/(2\pi^2)$.
}
\begin{figure}[t]
\hspace{0.2cm}
\centerline{\footnotesize\hspace{-0.2cm}$t = 0.01$ \hspace{4.6cm} $t = 0.5$  \hspace{4.5cm} $t = 1.0$ }

\hspace{0.1cm} 
\centerline{
	\rotatebox{90}{\hspace{1.5cm}\footnotesize Adaptive FTT }
		\includegraphics[width=0.32\textwidth]{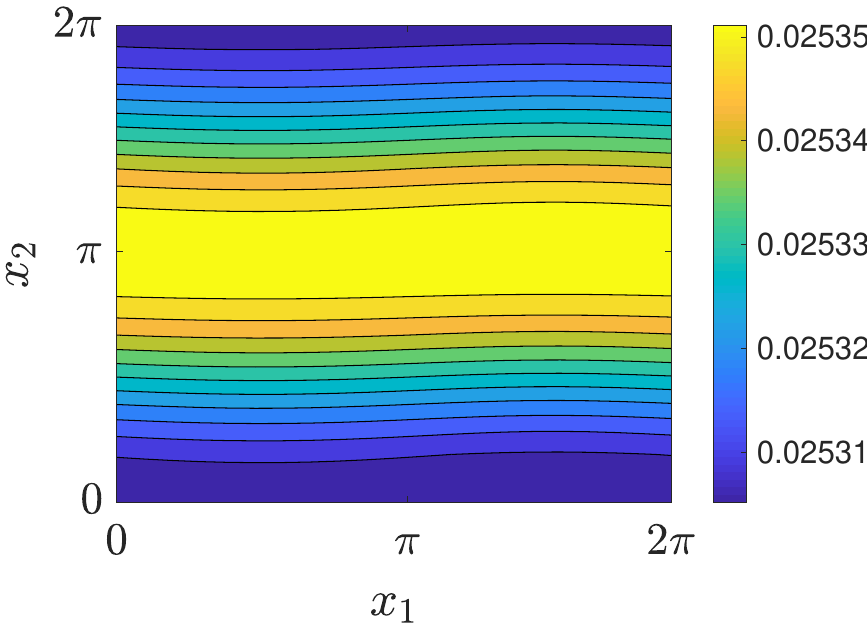}
		\hspace{.3cm}
		\includegraphics[width=0.32\textwidth]{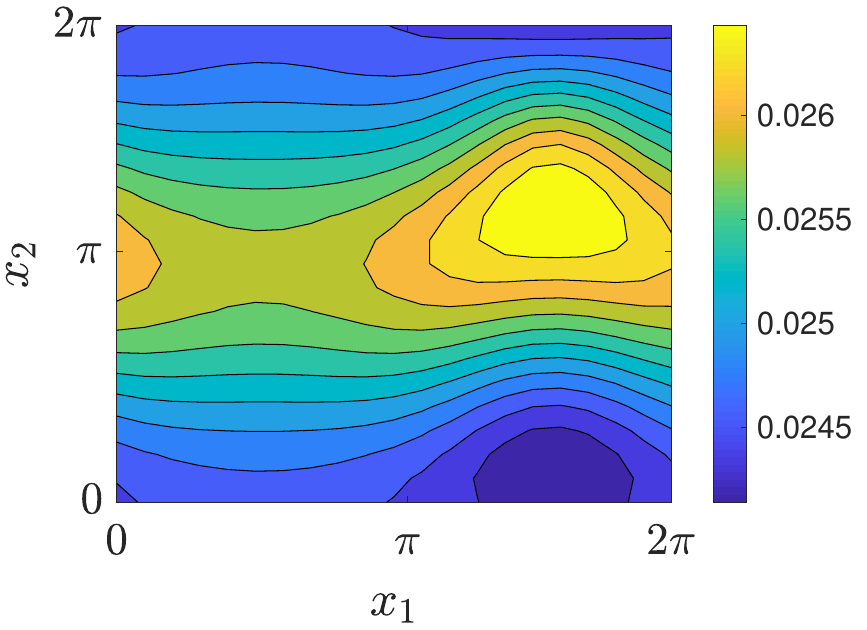}
		\hspace{.3cm}
		\includegraphics[width=0.32\textwidth]{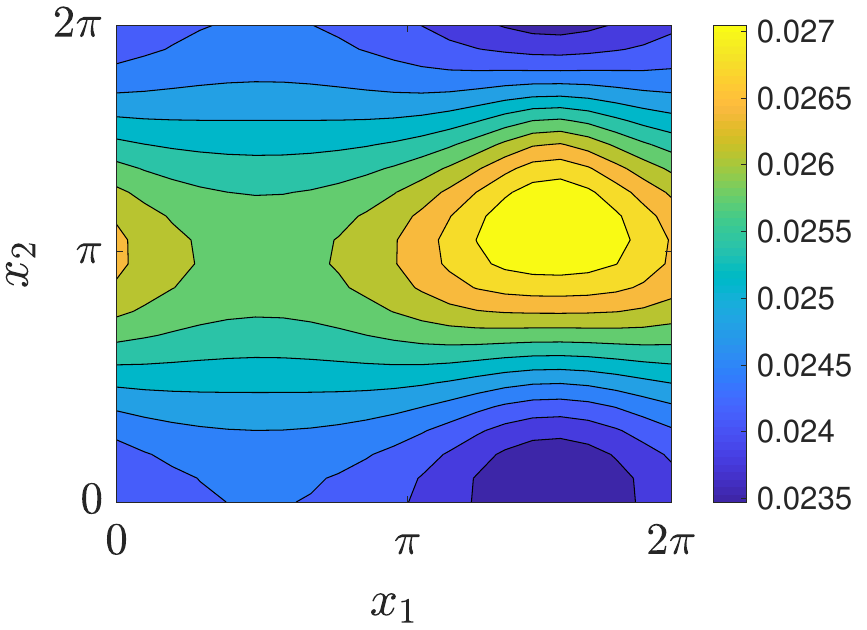}	
		\hspace{.2cm}
}
	
\hspace{0.1cm}	
\centerline{
\rotatebox{90}{\hspace{1.0cm}  \footnotesize  Full tensor product}
       \includegraphics[width=0.32\textwidth]{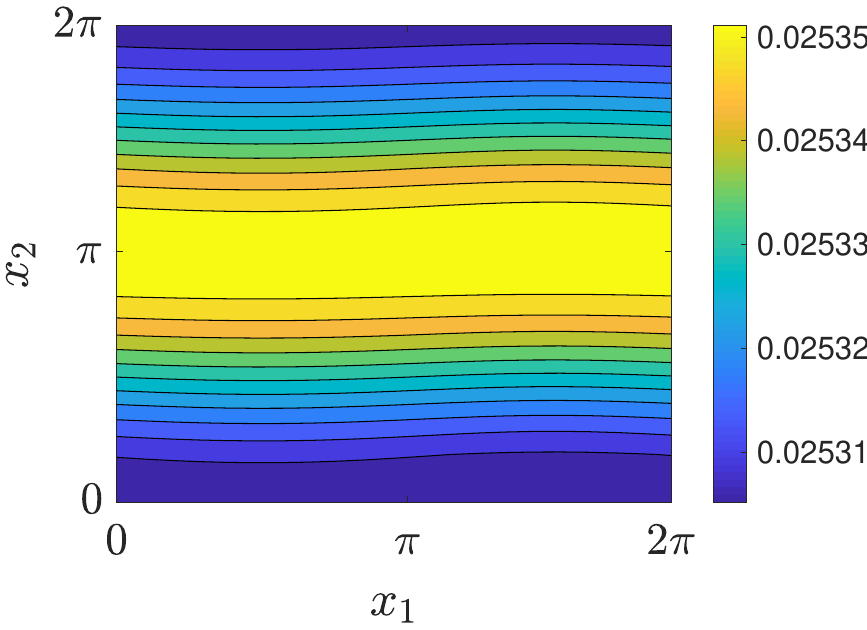}
       \hspace{.3cm}
	    \includegraphics[width=0.32\textwidth]{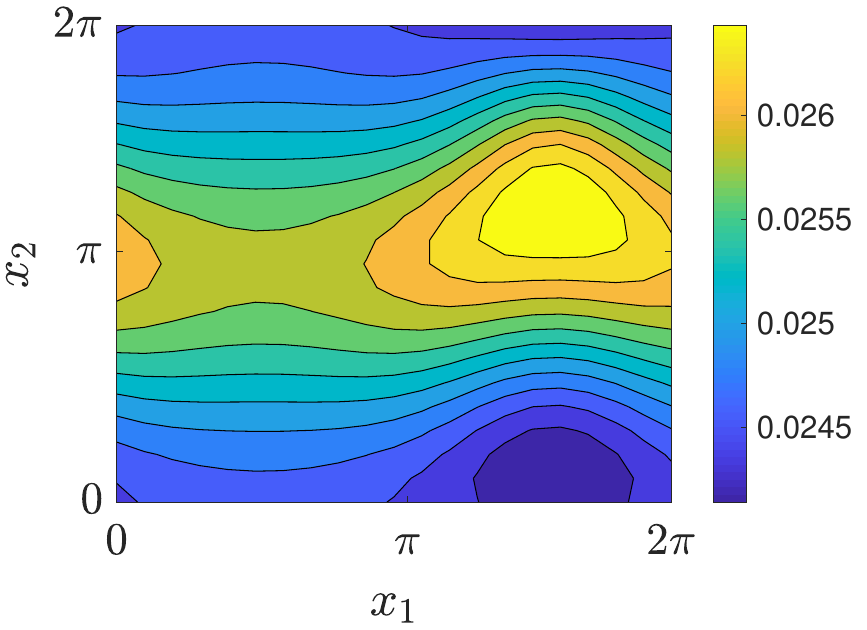}
	    \hspace{.3cm}
	    \includegraphics[width=0.32\textwidth]{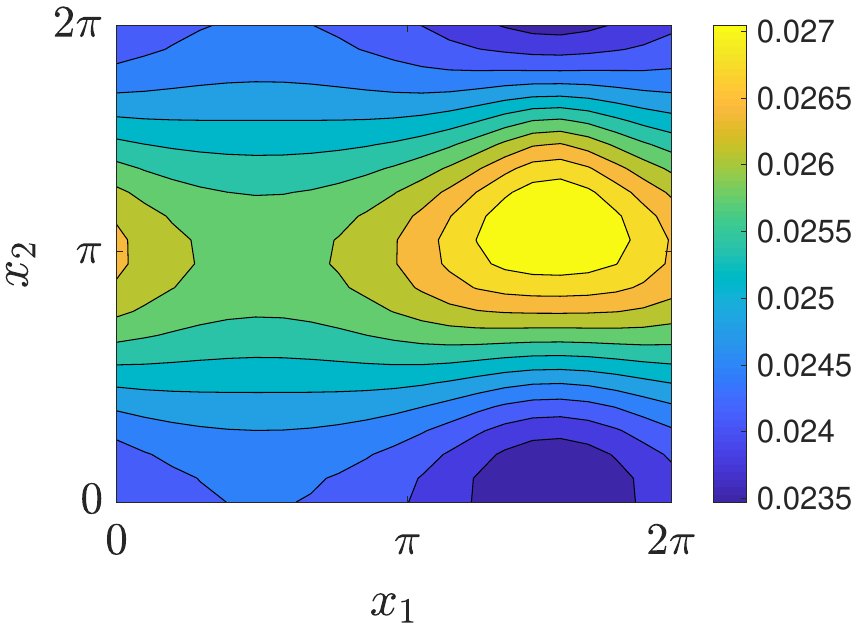}
	    \hspace{.2cm}
}

\hspace{0.0cm}	
\centerline{\rotatebox{90}{\hspace{1.2cm}  \footnotesize Pointwise error}
\hspace{0.1cm}
\includegraphics[width=0.33\textwidth,height=4.2cm]{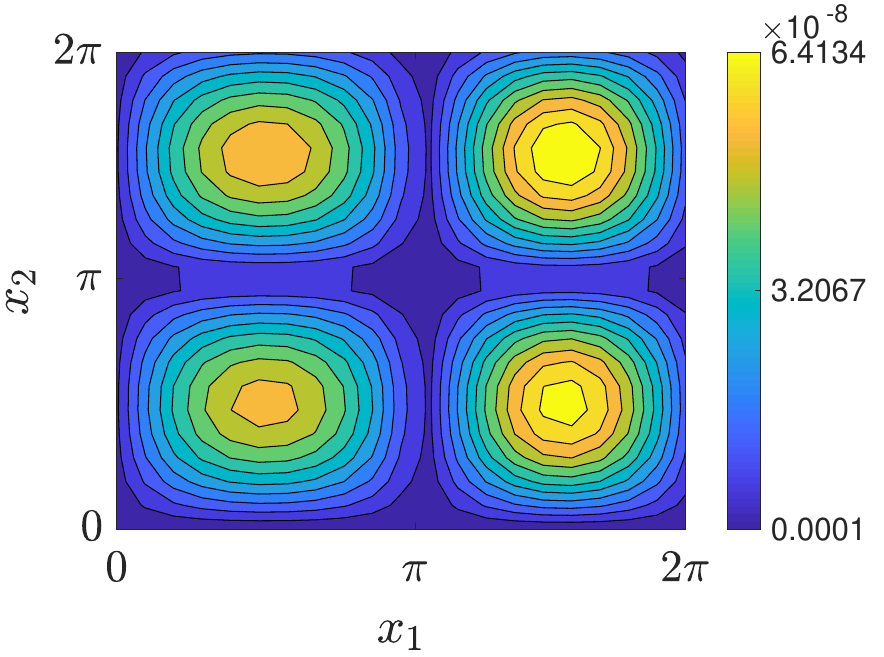}\hspace{0.2cm}
	    \includegraphics[width=0.33\textwidth,height=4.2cm]{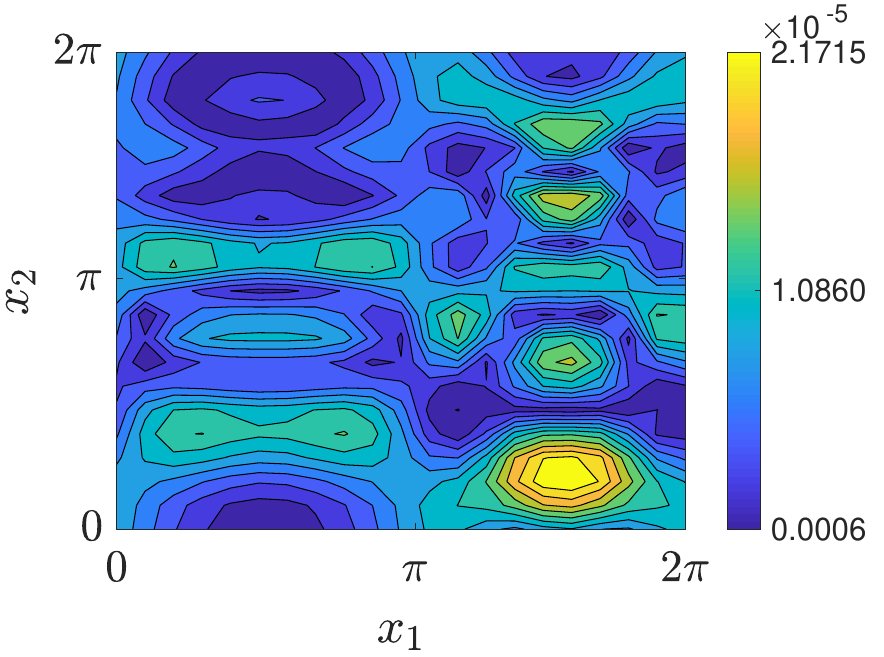}\hspace{0.2cm}
	    \includegraphics[width=0.33\textwidth,height=4.2cm]{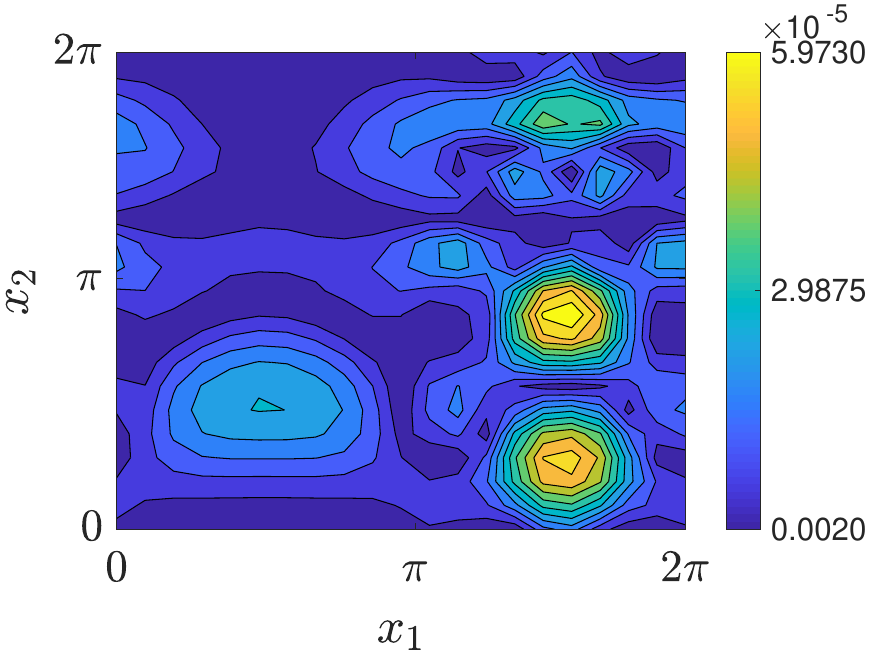}
}
\caption{Time snapshots of marginal PDF $p_{\bm r}(x_1,x_2,t)$  corresponding to the solution to the Fokker-Planck equation \eqref{Fokker_Planck}. 
We plot marginals computed with the rank-adaptive FTT integrator using 
$\epsilon_{\mathrm{inc}} = 10^{-4}$ (top row) and with 
the full tensor product Fourier pseudo-spectral method (middle row). 
We also plot the pointwise error between the two numerical solutions (bottom row). 
The initial condition is the FTT tensor \eqref{PDF_IC}.}
\label{fig:FP_solution_time_evolution}
\end{figure}
To obtain a benchmark solution with which to 
compare the rank-adaptive FTT solution, 
we solve the PDE \eqref{Fokker_Planck} 
using a Fourier pseudo-spectral method 
on the flat torus $\mathbb{T}^4$ with $21^4=194481$ 
evenly-spaced points. As before, the 
operator $\mathcal{L}$ is represented in terms 
of pseudo-spectral differentiation 
matrices \cite{spectral_methods_book}, 
and the resulting semi-discrete approximation 
(ODE system) is integrated with an explicit 
fourth-order Runge Kutta method 
using time step $\Delta t = 10^{-4}$. 
The numerical solution we obtained in this 
way is denoted by $p_{\text{ref}}(\bm x,t)$.  
We also solve the Fokker-Planck using the proposed 
rank-adaptive FTT method with first-order Lie-Trotter 
time integrator (section \ref{sec:splitting_integrator}) 
and normal vector thresholding (section \ref{sec:increase_rank}). 
We run three simulations all with time step 
$\Delta t = 10^{-4}$: one with no rank adaption, and two 
with rank-adaptation and normal component thresholds set 
to $\epsilon_{\mathrm{inc}} = 10^{-3}$ and 
$\epsilon_{\mathrm{inc}} = 10^{-4}$. 
In Figure \ref{fig:FP_solution_time_evolution} we 
plot three time snapshots of the two-dimensional 
solution marginal
\begin{equation}
\label{marg_pdf}
p(x_1,x_2,t) = \int_{0}^{2\pi} 
\int_{0}^{2 \pi} p(x_1,x_2,x_3,x_4,t) dx_3 dx_4
\end{equation}
computed with the rank-adaptive FTT integrator 
($\epsilon_{\mathrm{inc}} = 10^{-4}$) and the full 
tensor product pseudo-spectral method (reference solution). 
{In Figure \ref{fig:FP_4D_errors}(a) we compare the $L^2(\Omega)$ errors of the rank-adaptive method relative 
to the reference solution. It is seen that as we decrease the threshold the solution becomes more accurate.}
In Figure \ref{fig:FP_4D_errors}(b)
we plot the component of $\mathcal{L} {p_{\bm r}}$ 
normal to the tensor manifold, which is approximated using the 
two-point BDF formula \eqref{backwards_fd_1st_order}. 
Note that in the rank-adaptive FTT solution with thresholds 
$\epsilon_{\text{inc}}=10^{-3}$ and $\epsilon_{\text{inc}}=10^{-4}$ the solver 
performs both mode addition as well as mode removal.
This is documented in Figure \ref{fig:4D_FP_ranks}. 
{
The abrupt change in rank observed in Figure \ref{fig:4D_FP_ranks}(a)-(c) near time $t=0.4$ 
corresponding to the rank-adaptive solution with threshold $\epsilon{\text{inc}} = 10^{-4}$ is due 
to the time step size $\Delta t$ being equal to $\epsilon_{\text{inc}}$.
This can be justified as follows. 
Recall that the solution is first order accurate in $\Delta t$ and therefore the 
approximation of the component of $\mathcal{L} p_{\bm r}$ normal to the tensor manifold $\mathcal{M}_{\bm r}$ 
is first-order accurate in $\Delta t$. If we set $\epsilon_{\text{inc}} \leq \Delta t$, then the rank-adaptive scheme 
may overestimate the number of modes needed to achieve accuracy on the order of $\Delta t$. 
This does not affect the accuracy of the numerical solution due to the robustness 
of the Lie-Trotter integrator to over-approximation \cite{Lubich_2015}. 
Moreover we notice that the rank-adaptive scheme removes the unnecessary modes 
ensure that the 
tensor rank is not unnecessarily large 
(see section \ref{sec:decreasing_rank}).} 
In fact, the 
diffusive nature of the Fokker-Plank equation on the flat torus 
$\mathbb{T}^4$ yields relaxation to a statistical 
equilibrium state that depends on the drift and diffusion 
coefficients in \eqref{Fokker_Planck}. Such an equilibrium state 
may be well-approximated by a low-rank FTT tensor.

\begin{figure}[t]
\centerline{\footnotesize\hspace{0.5cm} (a) \hspace{7.1cm} (b) }
	\centering
		\includegraphics[width=0.45\textwidth]{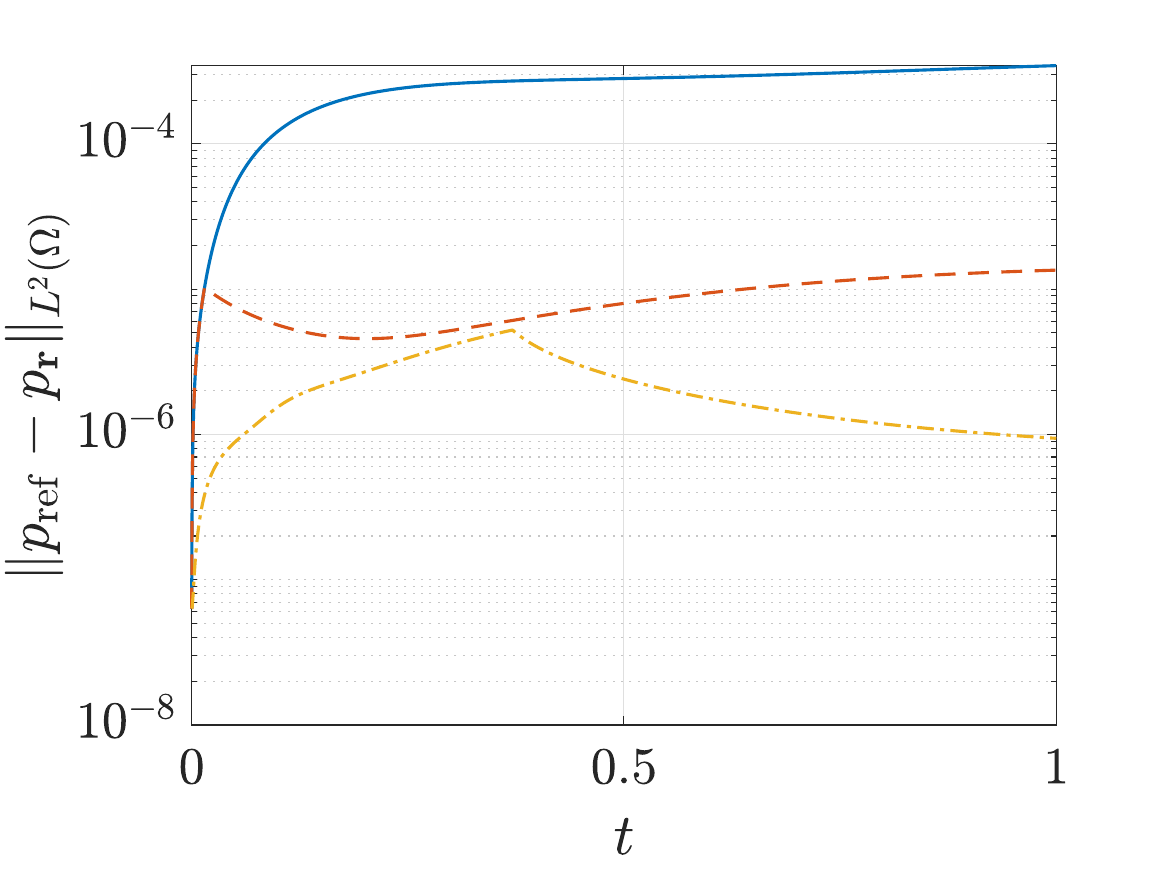}
		\includegraphics[width=0.45\textwidth]{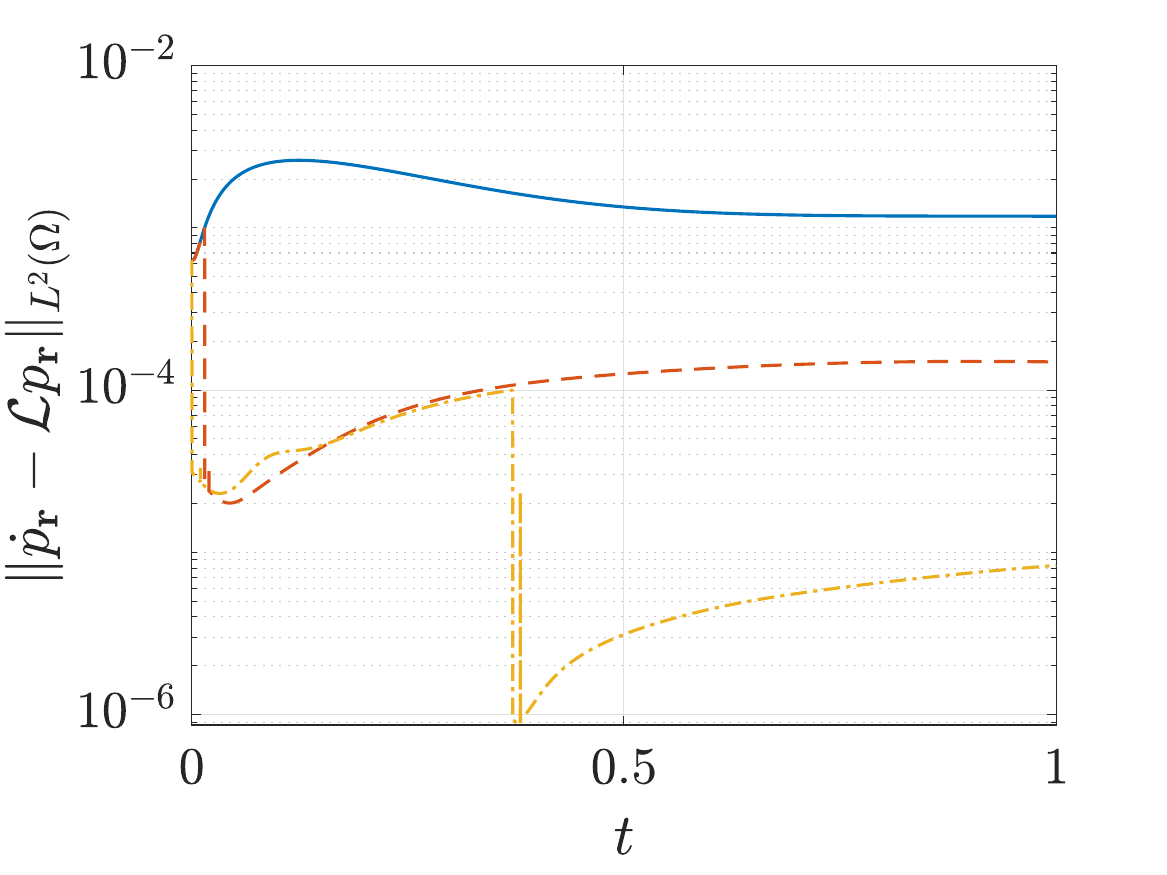} \\
		\hspace{0.8cm}\includegraphics[scale=0.5]{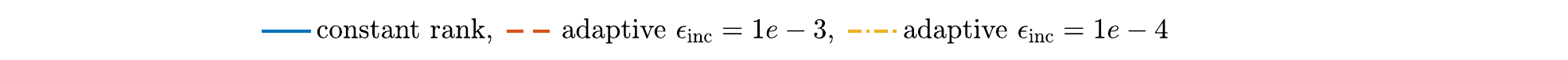}
\caption{
(a) The $L^2(\Omega)$ error of the FTT solution ${p_{\bm r}}(\bm x,t)$ 
relative to the benchmark solution $p_{\text{ref}}(\bm x,t)$ computed with a Fourier pseudo-spectral method on a tensor product grid. 
(b) Norm of the component of $\mathcal{L} {p_{\bm r}}$ normal to the tensor manifold (see Figure \ref{fig:normal component}). Such component is approximated a two-point BDF formula at each time step.}
\label{fig:FP_4D_errors}
\end{figure}

\begin{figure}[t]
\centerline{\footnotesize\hspace{0.25cm} (a) \hspace{5.1cm} (b) \hspace{5.1cm} (c) }
	\centering
		\includegraphics[width=0.33\textwidth]{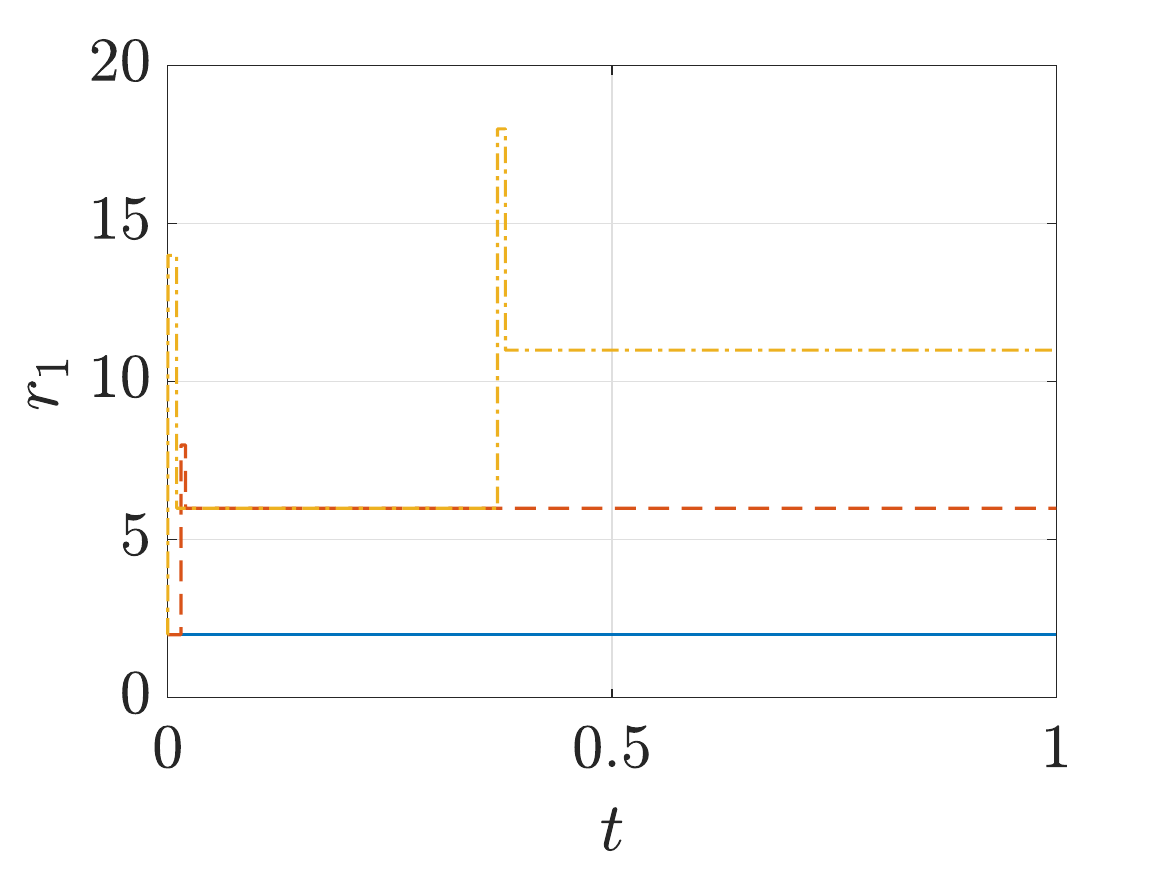}
		\includegraphics[width=0.33\textwidth]{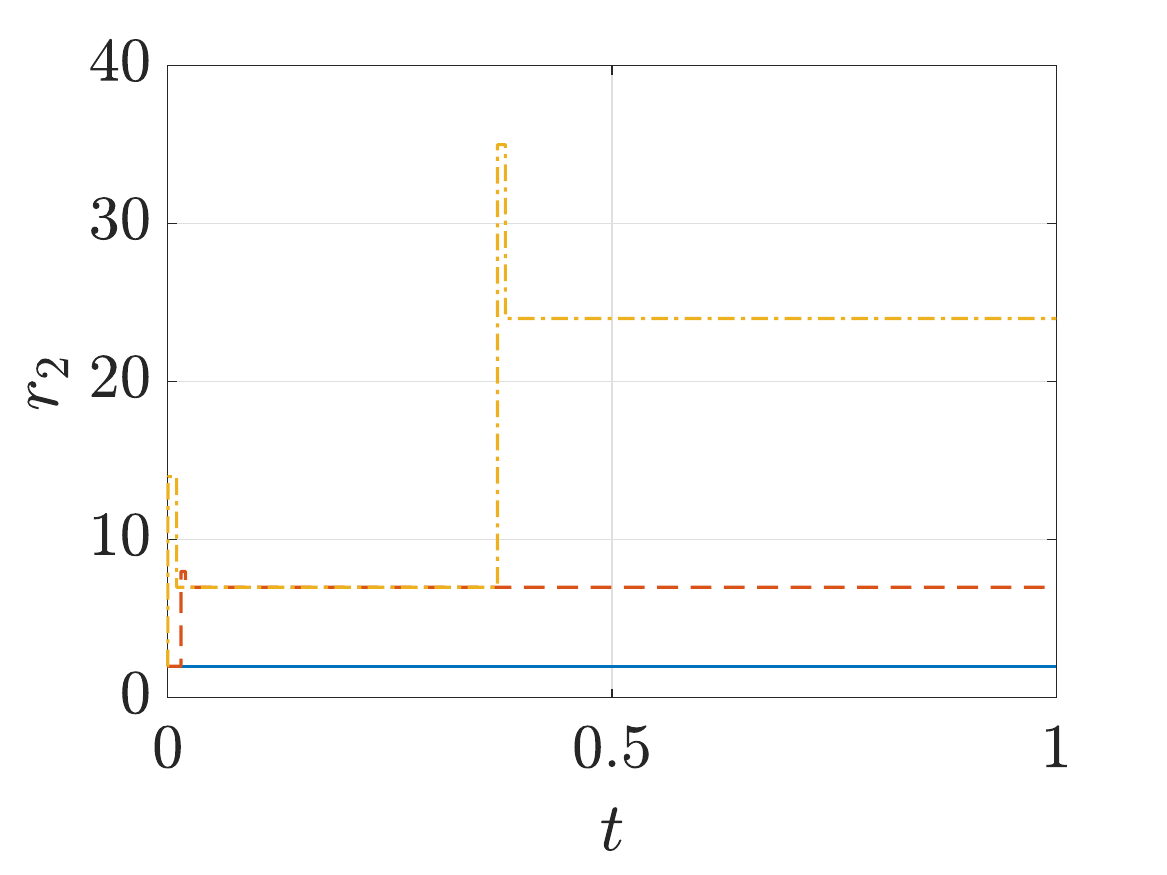}
		\includegraphics[width=0.33\textwidth]{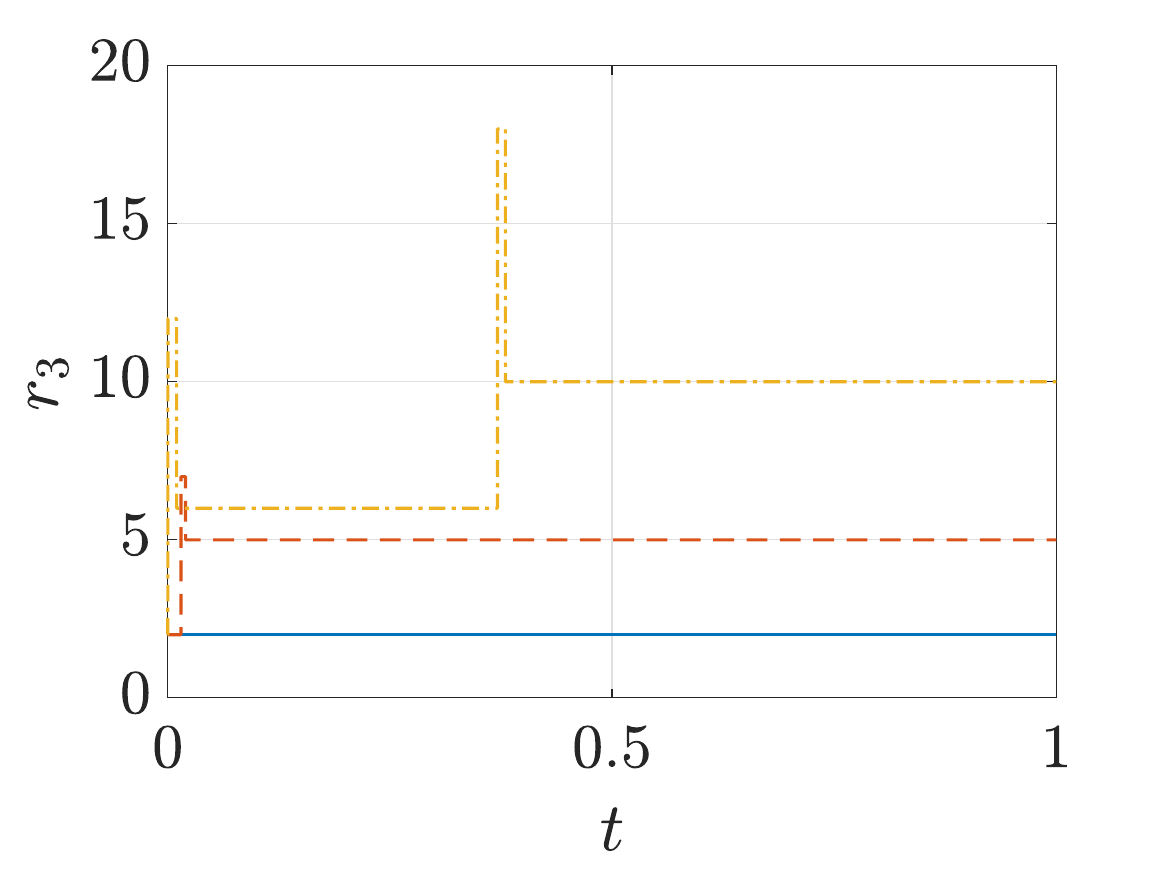}	\\
		\hspace{0.8cm}\includegraphics[scale=0.5]{legend_4D_FP-eps-converted-to.pdf}

\caption{Tensor rank $\bm r=[ 1 \, r_1 \, r_2\,r_3 \, 1]$ of adaptive FTT solution to 
the four dimensional Fokker-Planck equation \eqref{Fokker_Planck}.}
\label{fig:4D_FP_ranks}
\end{figure}

\section{Summary}
\label{sec:summary}

We presented a new rank-adaptive tensor 
method to integrate high-dimensional 
initial-boundary value problems for nonlinear PDEs.  
The new method  is based on functional tensor 
train (FTT) expansions 
\cite{Dektor_dyn_approx,OseledetsTT,Bigoni_2016}, operator 
splitting time integration \cite{Lubich_2016,Lubich_2018}, 
and a new rank-adaptive algorithm 
to add and remove tensor modes from the PDE solution 
based on thresholding the component of the 
velocity vector normal to the FTT tensor manifold. 
We tested the proposed new algorithm 
on three different initial/boundary value 
problems including a 2D variable-coefficient 
first-order linear PDE, a 2D Kuramoto-Sivashinsky 
equation, and a 4D Fokker-Planck equation. 
In all cases the adaptive FTT solution was 
compared to a benchmark numerical solution 
constructed with well-established numerical 
methods. 
The numerical results we obtained 
demonstrate that the proposed 
rank-adaptive tensor method 
is effective in controlling the temporal 
integration error, and outperforms known integration methods for 
multidimensional PDEs in 
terms of accuracy, robustness and 
computational cost. 
We also proved that the new method is 
consistent with recently proposed 
step-truncation algorithms \cite{Vandereycken_2019,
rodgers2020step-truncation,rodgers2020stability}
in the limit of small time steps.

\vspace{0.5cm}
\noindent
{\bf Acknowledgements} This research was supported 
by the U.S. Air Force Office of Scientific Research 
(AFOSR) grant FA9550-20-1-0174 and by the U.S. Army 
Research Office (ARO) grant W911NF-18-1-0309.

\bibliographystyle{plain}
\bibliography{bibliography_file}

\end{document}